\documentclass[12pt, a4paper]{article}
\usepackage{amsmath, srcltx}
\usepackage{amssymb}
\usepackage{amsfonts,cmtiup}

\usepackage{mathrsfs,cmtiup}
\usepackage{amsthm, eucal, eufrak}

\usepackage{srcltx}
\usepackage{amsfonts,amssymb,mathrsfs,amsmath,cmtiup}
\usepackage{amsthm, eucal, eufrak}

\newtheorem{theorem}{Theorem}[section]
\newtheorem{lemma}[theorem]{Lemma}
\newtheorem{proposition}[theorem]{Proposition}
\newtheorem{corollary}[theorem]{Corollary}

\theoremstyle{definition}
\newtheorem*{definition}{Definition}
\newtheorem{remark}[theorem]{Remark}
\newtheorem*{Index Convention}{Index Convention}
\newtheorem*{notation}{Notation}
\newtheorem*{definition0}{Definition of level 0}

\newtheorem*{definition3}{\boldmath Definition of level $t>0$}

\def\keywords#1{\par\medskip
\noindent\textbf{Key words.} #1}

\def\subjclass#1{{\renewcommand{\thefootnote}{}
\footnote{\emph{Mathematics Subject Classification (2010):} #1}}}

\begin{document}
\let\le=\leqslant
\let\ge=\geqslant
\let\leq=\leqslant
\let\geq=\geqslant
\newcommand{\e}{\varepsilon }
\newcommand{\g}{\gamma}
\newcommand{\F}{{\Bbb F}}
\newcommand{\N}{{\Bbb N}}
\newcommand{\Z}{{\Bbb Z}}
\newcommand{\Q}{{\Bbb Q}}
\newcommand{\R}{\Rightarrow }
\newcommand{\W}{\Omega }
\newcommand{\s}{\sigma }
\newcommand{\ka}{\varkappa }
\newcommand{\Lr}{\Leftrightarrow }
\newcommand{\al }{\alpha }
\newcommand{\w}{\omega }
\newcommand{\bt}{\begin{theorem}}
\newcommand{\et}{\end{theorem}}
\newcommand{\be}{\begin{equation}}
\newcommand{\ee}{\end{equation}}
\newcommand{\bc}{\begin{corollary}}
\newcommand{\ec}{\end{corollary}}
\newcommand{\bp}{\begin{proof}}
\newcommand{\ep}{\end{proof}}
\newcommand{\bl}{\begin{lemma}}
\newcommand{\el}{\end{lemma}}

\title{Finite groups and Lie rings with a metacyclic Frobenius group of automorphisms}

\markright{}

\author{{E.\,I.~Khukhro}\\ \small Sobolev Institute of Mathematics, Novosibirsk,
630\,090, Russia\\[-1ex] \small khukhro@yahoo.co.uk \\
{N.\,Yu.~Makarenko\footnote{The second author  was supported by the Russian Foundation for Basic Research, project no.~13-01-00505.}}\\ \small Universit\'{e} de Haute
Alsace, Mulhouse, 68093, France, and \\ \small  Sobolev Institute of Mathematics,
Novosibirsk, 630\,090, Russia
\\[-1ex] \small  natalia\_makarenko@yahoo.fr}

\date{}
\maketitle

\begin{center}{\it to
Victor Danilovich Mazurov on the occasion of his
70th birthday}
\end{center}

\begin{abstract}
Suppose that a finite group $G$ admits a Frobenius group of
automorphisms $FH$ of coprime order with cyclic kernel $F$ and
complement $H$ such that the fixed point subgroup $C_G(H)$ of
the complement is nilpotent of class $c$. It is proved that $G$
has a nilpotent characteristic subgroup of index bounded in terms of $c$,
$|C_G(F)|$, and $|F|$ whose nilpotency class is bounded in terms
of $c$ and $|H|$ only. This generalizes the previous theorem of the
authors and P.~Shumyatsky, where for the case of $C_G(F)=1$ the
whole group was proved to be nilpotent of $(c,|H|)$-bounded class.
Examples show that the condition of $F$ being cyclic is essential.
Results 
based on the classification  provide reduction to soluble groups. Then representation theory
arguments are used to bound the index of the Fitting subgroup. Lie
ring methods are used for nilpotent groups. A similar theorem on
Lie rings with a metacyclic Frobenius group of automorphisms $FH$
is also proved.
\end{abstract}


\section{Introduction}

Suppose that a finite group $G$ admits a Frobenius group of
automorphisms $FH$ of coprime order with cyclic kernel $F$ and
complement $H$. In a number of recent papers the structure of $G$
was studied under the assumption that the kernel acts without
nontrivial fixed points: $C_G(F)=1$. A fixed-point-free action of
$F$ alone was already known to imply many nice properties of $G$
(see more on this below). But the
`additional' action of the Frobenius complement $H$ suggested
another approach to the study of $G$. Namely, in the case $C_G(F)=1$, by Clifford's theorem
all $FH$-inva\-ri\-ant elementary abelian sections of $G$ are free
${\Bbb F}_pH$-modules (for various $p$),
and therefore it is natural to expect that many properties or
parameters of $G$ should be close to the corresponding properties
or parameters of $C_G(H)$, possibly also depending on $H$.
Prompted by Mazurov's problem 17.72 in Kourovka
Notebook \cite{kour}, several results of this nature were obtained
recently \cite{khu08, mak-shu10,khu10al, khu-ma-shu,
khu-ma-shu-DAN, shu-a4, shu-law, khu12ja,khu12al}, the properties
and parameters in question being the order, rank, Fitting
height, nilpotency class, and exponent. In particular, it
was proved in \cite{khu-ma-shu} that if $C_G(H)$ is nilpotent of
class $c$, then $G$ is nilpotent of $(c,|H|)$-bounded class (a special case of this result solving
part (a) of Mazurov's problem was proved earlier by the second author and Shumyatsky \cite{mak-shu10}).
Henceforth we write for brevity, say,
``$(a,b,\dots )$-bounded'' for ``bounded above by some function
depending only on $a, b,\dots $''.

An important next step is considering finite groups $G$ with a Frobenius group of automorphisms $FH$
in which the kernel $F$ no longer acts
fixed-point-freely but has a relatively small number of fixed points. Then it is natural to strive for
similar restrictions, in terms of the complement $H$ and its fixed points $C_G(H)$,
for a subgroup of index bounded in terms of $|C_G(F)|$ and other parameters:
``almost fixed-point-free'' action of $F$ implying that $G$ is ``almost''
as good as when $F$ acts fixed-point-freely.
Such restrictions for the order and rank of $G$ were recently obtained
in \cite{khu13}. In the present paper we deal with the nilpotency
class assuming that $FH$ is a metacyclic Frobenius group. Examples
in \cite{khu-ma-shu} show that such results cannot be obtained for
non-metacyclic $FH$, even in the case $C_G(F)=1$.

\begin{theorem} \label{t-g}
Suppose that a finite group $G$ admits a Frobenius group of
automorphisms $FH$ of coprime order with cyclic kernel $F$ and
complement $H$ such that the fixed-point subgroup $C_G(H)$ of the
complement is nilpotent of class $c$. Then $G$ has a
nilpotent characteristic subgroup of index bounded in terms of
$c$, $|C_G(F)|$, and $|F|$ whose nilpotency
class is bounded in terms of $c$ and $|H|$ only.
\end{theorem}

In the proof, reduction to soluble groups is given by results 
based on the classification (\cite{hart} or \cite{wang-chen}). Then
representation theory arguments are used to bound the index of the
Fitting subgroup, thus reducing the proof to the case of a
nilpotent group $G$. We state separately the corresponding
Theorem~\ref{t-n}, since it gives a better bound
for the index of the Fitting subgroup and does not require the
Frobenius group $FH$ to be metacyclic.

For nilpotent groups, a Lie ring method is used. A similar
theorem on Lie rings is also proved, although
its application to the case of nilpotent group in Theorem~\ref{t-g} is not straightforward
and requires additional efforts.

\begin{theorem} \label{t-l}
Suppose that a finite Frobenius group
$FH$
with cyclic kernel $F$ and
complement $H$ acts by automorphisms on a
Lie ring $L$ in whose ground ring $|F|$ is invertible.
If the fixed-point subring $C_L(H)$ of the
complement is nilpotent of class $c$ and the fixed-point subring of the kernel $C_L(F)$ is finite of order $m$,
then $L$ has a nilpotent Lie subring whose index in the additive group of $L$ is bounded in terms of
$c$, $m$, and $|F|$ and whose nilpotency class is bounded in terms of $c$ and $|H|$ only.
\end{theorem}

The functions bounding the index and nilpotency class in the theorems
can be estimated from above explicitly, although we do not write out these estimates here.
For Lie algebras we do not need the condition that $|F|$ be invertible.

\begin{corollary} \label{c-l}
Suppose that a finite Frobenius group
$FH$ with cyclic kernel $F$ and
complement $H$ acts by automorphisms on a
Lie algebra $L$ in such a way that the fixed-point subalgebra $C_L(H)$
is nilpotent of class $c$ and the fixed-point subalgebra $C_L(F)$ has  finite dimension $m$.
Then $L$ has a nilpotent Lie subalgebra of finite codimension bounded in terms of
$c$, $m$, and $|F|$ and whose nilpotency class is bounded in terms of $c$ and $|H|$ only.\end{corollary}

Earlier in \cite{mak-khu13} we proved this
under the additional condition that the characteristic of $L$ be coprime to $|H|$.

We now discuss in more detail the context of the two parts of the proof of Theorem~\ref{t-g}, which are
quite different, the first about bounding the index
of the Fitting subgroup by methods of representation theory, and the second about bounding the nilpotency
class of a subgroup of bounded index by Lie ring methods.

Let $G$ be a finite (soluble) group $G$ admitting a
soluble group of automorphisms $A$ of coprime order.
Connections between the Fitting heights of $G$ and $C_G(A)$,
depending also on the number $\alpha (A)$ of prime factors in
$|A|$, were first established by Thompson \cite{th2} and later improved
by various authors, including linear bounds by Kurzweil
\cite{kurz} and Turull \cite{tu}. The Hartley--Isaacs theorem
\cite{ha-is} (using Turull's result \cite{tu}) says that
$|G:F_{2\alpha (A)+1}|$ is bounded in terms of $\alpha (A)$ and
$|C_G(A)|$. These results can of course be applied both to the action of
$F$ and of $H$ in our Theorem~\ref{t-g}. But our conclusion is in a sense much
stronger, bounding the index of a nilpotent subgroup, rather than of a subgroup of
Fitting height depending on $\alpha (F)$ or $\alpha (H)$. Of course, this is
due to the stronger hypotheses of combined actions of the kernel and
the complement, neither of which alone is sufficient.

Now suppose that the group $G$ is already nilpotent and admits an
(almost) fixed-point-free group of automorphisms $A$. Then
further questions arise about bounding the nilpotency class or
the derived length of $G$ (or of a subgroup of bounded index).
Examples show that such bounds can only be achieved if $A$ is
cyclic. By Higman's theorem \cite{hi} a (locally) nilpotent group with a
fixed-point-free automorphism of prime order $p$ is nilpotent of
$p$-bounded class. This immediately
follows from the Higman--Kreknin--Kostrikin theorem
\cite{hi,kr,kr-ko} saying that a Lie ring with a
fixed-point-free automorphism of prime order $p$ is nilpotent of
$p$-bounded class.

The first author \cite{kh1, kh2}
proved that if a periodic (locally) nilpotent group $G$ admits an
automorphism $\varphi$ of prime order $p$ with $m=|C_G(\varphi )|$ fixed
points, then $G$ has a nilpotent subgroup of $(m,p)$-bounded index
and of $p$-bounded class. (The result was later extended by Medvedev
\cite{me} to not necessarily periodic locally nilpotent groups.)
This group result was also based on a
similar theorem on Lie rings in \cite{kh2}, albeit also on
additional arguments, as in general there is no good
correspondence between subrings of the associated Lie ring and
subgroups of a group.
The proofs in \cite{kh2}
were based on a method of graded centralizers; this method was later developed by the
authors \cite{khmk4, mk05,
khmk1, khmk2,
khmk3,  khmk5} in further studies of almost fixed-point-free
automorphisms of Lie rings and nilpotent groups. It is this method that we also use
in the proofs of both the Lie ring Theorem~\ref{t-l} and the
nilpotent case of the group Theorem~\ref{t-g}.

There is apparent similarity between the relation of the above-mentioned theorem
on ``almost fixed-point-free'' automorphism of prime order to the ``fixed-point-free''
Higman--Kreknin--Kostrikin theorem and the relation of the results
of the present paper on a Frobenius group of automorphisms with
``almost fixed-point-free'' kernel to the
Khukhro--Makarenko--Shumyatsky theorem \cite{khu-ma-shu} on a Frobenius group of
automorphisms with fixed-point-free kernel: in both cases a bound
for the nilpotency class of the whole group is replaced by a bound
for the nilpotency class of a subgroup of bounded index. In fact,
this similarity goes deeper than just the form of the results: the
method of proof of the Lie ring Theorem~\ref{t-l} and of the
nilpotent case of the group Theorem~\ref{t-g} is a modification of
the aforementioned method of graded centralizers used in
\cite{kh2}. In both cases, the previous nilpotency results are used
as certain combinatorial facts about Lie rings with finite cyclic
grading, which give rise to certain transformations of commutators.
The HKK-transformation in \cite{kh2} was based on the Higman--Kreknin--Kostrikin theorem, and
in the present paper we use the KMS-transformation based on the Khukhro--Makarenko--Shumyatsky theorem \cite{khu-ma-shu},
combined with the machinery of the method of graded centralizers, with certain modifications.

In the present paper the cyclic group of automorphisms $F$ is of
arbitrary (composite) order. Recall that it is still an open
problem to bound the derived length of a finite group with a
fixed-point-free automorphism. So far this is known only in the above-mentioned case of
automorphism of prime order (and of order 4 due to Kov\'acs). The
problem is already reduced to nilpotent groups, and there is
Kreknin's theorem \cite{kr} giving bounded solubility of a Lie ring
with a fixed-point-free automorphism, but the existing Lie ring
methods cannot be used for bounding the derived length in general.
The authors \cite{khmk5} also proved almost solubility of Lie rings and
algebras admitting an almost regular automorphism of finite order, with
bounds for the derived length and codimension of a soluble subalgebra,
but for groups even the fixed-point-free case remains open. The latter result can
be applied to the Lie ring in Theorem~\ref{t-l}, but we need $(c,|H|)$-bounded nilpotency of a subring,
rather than $|F|$-bounded solubility.
 It is the combined actions of the kernel and
 the complement that have to be used here, neither of which alone is sufficient.

There remain several open problems about groups $G$ (and Lie
rings) with a Frobenius group of automorphisms $FH$ (with kernel
$F$ and complement $H$). For example, even in the case of a
2-Frobenius group $GFH$ (when $GF$ is also a Frobenius group),
Mazurov's question 17.72(b) remains open: is the exponent of $G$
bounded in terms of $|H|$ and the exponent of $C_G(H)$? Other open questions
in the case $C_G(F)=1$ include bounding the derived length of $G$ in terms of
that of $C_G(H)$ and $|H|$. Such questions are already reduced to nilpotent groups, since
it was proved in \cite{khu12ja} that then
the Fitting height of $G$ is equal to the Fitting height of
$C_G(H)$.

One notable difference of the results of
the present paper from the previous results in the case $C_G(F)=1$ is
that we impose the additional condition that the order of $G$ and
$FH$ be coprime. Although Hartley's theorem \cite{hart} 
would still provide
reduction to soluble groups without the coprimeness condition,
there are further difficulties that for now remain unresolved.
Note, for example, that it is still unknown if the Fitting height
of a finite soluble group admitting an automorphism of order $n$
with $m$ fixed points is bounded in terms of $m$ and $n$ (in the
coprime case even a better result is a special case
of the Hartley--Isaacs theorem \cite{ha-is}).

\section{Almost nilpotency}\label{s-n}

In this section we prove the ``almost nilpotency'' part of
Theorem~\ref{t-g}. It makes sense to state a separate theorem, as
the bound for the index of the Fitting subgroup depends only on
$|C_G(F)|$ and $|F|$. (Dependence on the nilpotency class $c$ of
$C_G(H)$ appears in addition in Theorem~\ref{t-g}, where a
nilpotent subgroup of $(c,|H|)$-bounded class is required.)
Moreover, in the following theorem, $FH$ is an arbitrary,
not necessarily metacyclic, Frobenius
group.

\begin{theorem} \label{t-n}
Suppose that a 
finite group $G$ admits a Frobenius group
of
automorphisms $FH$ of coprime order with
kernel $F$ and
complement $H$ such that the fixed-point subgroup $C_G(H)$ of the
complement is nilpotent. Then the index of the Fitting subgroup $F(G)$ is
bounded in terms of
$|C_G(F)|$ and $|F|$.
\end{theorem}

By the result of Wang and Chen \cite{wang-chen} 
based on the classification (applied to the coprime action of $H$ on $G$), the group $G$ is soluble.
Further proof in some parts resembles the proof of
\cite[Theorem~2.7(c)]{khu-ma-shu} and \cite[Theorem~2.1]{khu12ja}, where the case of $C_G(F)=1$ was considered.
But some arguments in \cite{khu-ma-shu} do not work because a certain section $Q$ here cannot be
assumed to be abelian, and some arguments in \cite{khu12ja} cannot be
applied as $F$ is no longer fixed-point-free everywhere. Instead,
an argument in \cite{khu10al} is adapted to our situation (although the result of \cite{khu10al} was superseded by
\cite{khu12ja}).

We begin with some preliminaries. Suppose that a group $A$
acts by automorphisms on a finite group $G$ of coprime order:
$(|A|,|G|)=1$. For every prime $p$, the group $G$ has an $A$-invariant Sylow $p$-subgroup.
The fixed points of the induced action of $A$
on the quotient $G/N$ by an $A$-invariant normal subgroup are
covered by fixed points of $A$ in $G$, that is, $ C_{G/N}(A)=C_G(A)N/N$.
A similar property also holds for a finite group $A$ of linear
transformations acting on a vector space over a field of
characteristic coprime to $|A|$.
These well-known properties of coprime action will be used without special references.

The following lemma is a consequence of Clifford's theorem.

\begin{lemma}[{\cite[Lemma~2.5]{khu-ma-shu}}]\label{l-free}
If a Frobenius group $FH$ with kernel $F$ and complement $H$ acts
by linear transformations on a vector space $V$ over a field $k$
in such a way that $C_V(F)=0$, then $V$ is a free $kH$-module.
\end{lemma}

It is also convenient to use the following theorem of Hartley and Isaacs~\cite{ha-is}.

\begin{theorem}[{\cite[Theorem~B]{ha-is}}] For an arbitrary finite
group $A$ there exists a number $ \delta (A)$
depending only on $A$ with the following property. Let $A$ act
on~$G$, where $G$ is a finite soluble group such that
$(|G|,|A|)=1$, and let $k$ be any field of characteristic not
dividing~$|A|$. Let $V$ be any irreducible $kAG$-module and let
$S$ be any $kA$-module that appears as a component of the
restriction~$V_A$. Then
 $\dim _kV\leqslant \delta (A) m_S$, where $
m_S$ is the multiplicity of $S$ in~$V_A$.
\end{theorem}

The following is a key proposition in the proof of
Theorem~\ref{t-n}. Note that here $C_G(H)$ is not assumed to be
 nilpotent, as a stronger assertion is needed for induction on $|H|$ to work.

\begin{proposition}\label{p1}
Let $G$ be a finite group admitting a Frobenius group of
automorphisms $FH$ of coprime order with
kernel $F$ and complement $H$. Suppose that
$V=F(G)=O_p(G)$ is an elementary abelian $p$-group
such that $C_V(F)=1$ and $G/V$ is a $q$-group. Then $F(C_G(H))\leq V$.
\end{proposition}

\bp
Let $Q$ be an $FH$-invariant Sylow $q$-subgroup of $G$.
Suppose the opposite and choose
a nontrivial element $c\in Q\cap Z(C_G(H)))$.
Note that $c$ centralizes $C_V(H)$ but acts nontrivially on $V$.
Our aim is a contradiction arising from these assumptions.

Consider $\langle c^{HF}\rangle=\langle c^F\rangle$, the minimal
$FH$-inva\-ri\-ant subgroup containing~$c$. We can assume that
\begin{equation}\label{zam-c}
Q=\langle
c^F\rangle.
\end{equation}

We regard $V$ as an
${\Bbb F}_pQFH$-module. At the same time we reserve the right to regard $V$ as a normal subgroup of the semidirect
product $VQFH$. For example, we may use the commutator notation: the subgroup
$[V,Q]=\langle [v,g]\mid v\in V,\;g\in Q\rangle$ coincides with the subspace spanned by
$\{-v+\nobreak vg\mid v\in V,\;g\in Q\}$. We also keep using the
centralizer notation for fixed points, like $C_V(H)=\{v\in V\mid vh=v\text{ for all } h\in H\}$,
and for kernels, like
$C_Q(Y)=\{x\in Q\mid yx=y\text{ for all }y\in Y\}$ for a subset $Y\subseteq V$.

We now extend the ground field to a finite field
$k$ that is a splitting field for $QFH$ and obtain a $kQFH$-module
$\widetilde V=V\otimes _{{\Bbb F}_p}k$.
Many of the above-mentioned properties of $V$ are inherited by ${\widetilde{V}}$:
\vskip1ex
\noindent (V1) \ $\widetilde{V}$ is a faithful $kQ$-module;

\noindent (V2) \ $c$ acts trivially on $C_{\widetilde{V}}(H)$;

\noindent (V3) \ $C_{\widetilde{V}}(F)=0$.

\vskip1ex
Our aim is to show that $c$ centralizes $\widetilde V$, which will contradict (V1).

Consider an unrefinable series of $kQFH$-sub\-modules
\begin{equation}\label{riad2}
\widetilde{V}=V_1 >V_2>\dots >V_n> V_{n+1}=0.
\end{equation}
Let $W$ be one of the factors of this series; it is a nontrivial
irreducible $kQFH$-module. If $c$ acts trivially on
every such $W$, then $c$ acts trivially on
$\widetilde V$, as the order of $c$ is coprime to the
characteristic $p$ of the field $k$ --- this contradicts (V1).
Therefore in what follows we assume that $c$ acts nontrivially on
$W$.

The following properties hold for $W$:

\vskip1ex
\noindent (W1) \ $c$ acts nontrivially on $W$;

\noindent(W2) \ $c$ acts trivially on $C_{W}(H)$;

 \noindent(W3) \ $C_{W}(F)=0$;

\noindent(W4) \ $W$ is a free $kH$-module.
\vskip1ex

Indeed, property (W1) has already been mentioned.
Property (W2) follows from (V2) since $C_{\widetilde{V}}(H)$ covers
$C_{W}(H)$. Property (W3) follows from (V3) since $C_{\widetilde{V}}(F)$ covers
$C_{W}(F)$. Property (W4) follows from (W3) by Lemma~\ref{l-free}.

We shall need the following elementary remark.

\begin{lemma} \label{l-faith}
Let $FH$ be a
Frobenius group with kernel $F$ and complement~$H$. In any
action of $FH$ with nontrivial action of $F$ the complement $H$
acts faithfully.\qed
\end{lemma}

 The following lemma will be used repeatedly in the
proof.

\begin{lemma} \label{c-triv-free}
Suppose that $M=\bigoplus _{h\in H}M_h$ is a free $kH$-sub\-module of $W$, that is, the subspaces $M_h$
form a regular $H$-orbit: $M_{h_1}h_2=M_{h_1h_2}$ for $h_1,h_2\in H$. If the element $c$ leaves invariant each
of the $M_h$, then $c$ acts trivially on $M$.
\end{lemma}

\begin{proof}
The fixed points of $H$ in $M$ are the diagonal elements
$\sum _{h\in H}mh$ for any $m$ in $M_1$, where $mh\in M_h$. Since
$c$ acts trivially on every such sum by property~(W2) and leaves invariant
every direct summand $M_h$, it must act trivially on each $mh$.
Clearly, the elements $mh$ run over all elements in all the summands $M_h=M_1h$, $h\in H$.
\end{proof}

We now apply Clifford's theorem and consider the decomposition
$
W=W_1\oplus \dots \oplus W_t
$
of $W$ into the direct sum of the
Wedderburn components $W_i$ with respect to~$Q$. We consider the
transitive action of $FH$ on the set $\Omega=\{W_1,\dots ,W_t\}$.

\begin{lemma} \label{c-triv-reg}
The element $c$ acts trivially on the sum of components in any regular $H$-orbit in $\Omega$.
\end{lemma}

\begin{proof}
This follows from Lemma~\ref{c-triv-free}, because the sum of components in a regular $H$-orbit in $\Omega$ is
obviously a free $kH$-sub\-module.
\end{proof}

Note that $H$ transitively permutes the $F$-orbits in $\Omega$.
Let $\Omega _1=W_1^F$ be one of these $F$-orbits and let $H_1$ be
the stabilizer of $\Omega _1$ in $H$ in the action of $H$ on
$F$-orbits. If $H_1=1$, then all the $H$-orbits in $\Omega$ are
regular, and then $c$ acts trivially on $W$ by
Lemma~\ref{c-triv-reg}. This contradicts our assumption (W1) that
$c$ acts nontrivially on $W$.
Thus, we assume that $H_1\ne 1$.

\begin{lemma} \label{orb-h1}
The subgroup $H_1$ has exactly one
non-regular orbit in~$\Omega _1$ and this orbit is a fixed point.
\end{lemma}

\begin{proof}
Let $\overline F$ be the image of $F$
in its action on $\Omega _1$. If $\overline F=1$, then $\Omega
_1=\{W_1\}$ consists of a single Wedderburn component, and the lemma holds.

Thus, we can assume that $\overline F\ne 1$, and
$\overline{F}H_1$ is a Frobenius group with complement~$H_1$. By
Lemma~\ref{l-faith} the subgroup $H_1$ acts faithfully on $\Omega_1$ and we use
the same symbol for it in regard of its action on $\Omega_1$.

Let $S$ be the stabilizer of the point $W_1 \in \Omega _1$ in
$\overline{F}H_1$. Since $|\Omega _1|=|\overline F:\overline F\cap
S|=|\overline F H_1:S|$ and the orders $|\overline F|$ and $|H_1|$
are coprime, $S$ contains a conjugate of $H_1$; without loss of
generality (changing $W_1$ and therefore $S$ if necessary) we
assume that $H_1\leq S$. We already have a fixed point $W_1$ for
$H_1$. It follows that $H_1$ acts on $\Omega _1$ in the same way
as $H_1$ acts by conjugation on the cosets of the stabilizer of $W_1$ in $\overline F$. But in a Frobenius group
no non-trivial element of a complement can fix a non-trivial coset of a subgroup of the
kernel. Otherwise there would exist such an element of prime order and, since this element is fixed-point-free
on the kernel, its order would divide the order of that coset and therefore the order of the kernel, a contradiction.
\end{proof}

We now consider the $H$-orbits in $\Omega$. Clearly, the $H$-orbits of elements of regular $H_1$-orbits
in $\Omega_1$ are regular $H$-orbits. Thus, by Lemma~\ref{orb-h1} there is exactly one non-regular
$H$-orbit in $\Omega$ --- the $H$-orbit of the fixed point $W_1$ of $H_1$ in $\Omega
_1$. Therefore by Lemma~\ref{c-triv-reg} we obtain the following.

\begin{lemma} \label{c-triv-out}
The element $c$ acts trivially on all the Wedderburn components
$W_i$ that are not contained in the $H$-orbit of~$W_1$.\qed
\end{lemma}

Therefore, since $c$ acts nontrivially on $W$,
it must be nontrivial on the sum over the $H$-orbit of $W_1$.
Moreover, since $c$
commutes with $H$, the element $c$ acts in the same way --- and therefore nontrivially --- on all the components in
the $H$-orbit of $W_1$.
In particular, $c$ is nontrivial on $W_1$.

We employ induction on $|H|$.
In
the basis of this induction, $|H|$ is a prime, and either $H_1=1$,
which gives a contradiction as described above, or $H_1=H$, which
is the case dealt with below.

First suppose that
$H_1\ne H$. Then we consider $U=\bigoplus _{f\in F}W_1f$, the sum of components in $\Omega _1$,
which is a $kQFH_1$-module.

\begin{lemma} \label{c-triv-cu}
The element $c$ acts trivially on $C_U(H_1)$.
\end{lemma}

\begin{proof}
Indeed, $c$ acts trivially on the sum over any
regular $H_1$-orbit, because such an $H_1$-orbit is a part of a regular $H$-orbit,
on the sum over which $c$ acts trivially by Lemma~\ref{c-triv-reg}.
By Lemma~\ref{orb-h1} it remains to show that $c$ is trivial on $C_{W_1}(H_1)$.

For
$x\in C_{W_1}(H_1)$ and some right transversal $\{t_i\mid 1\leq i\leq
|H:H_1|\}$ of $H_1$ in $H$ we have $\sum _{i} xt_i\in C(H)$. Indeed, for
any $h\in H$ we have $\sum _{i} xt_ih=\sum _{i} xh_{1i}t_{j(i)}$
for some $h_{1i}\in H_1$ and some permutation of the same transversal $\{t_{j(i)}\mid 1\leq i\leq
|H:H_1|\}$, and the latter sum is equal to $\sum _{i} xt_{j(i)}=\sum _{i} xt_{i}$ as $xh_{1i}=x$ for all $i$.
Since $c$ acts trivially on $\sum _{i} xt_{i}\in C(H)$ by property~(W2), it must also act trivially
on each summand, as they are in different $c$-inva\-ri\-ant components; in particular, $xc=x$.
\end{proof}

In order to use induction on $|H|$, we consider the additive group
$U$ on which the group $ Q F H_1$ acts as a group of
automorphisms. The action of $Q$ and $F$ may not be faithful, but
the action of $H_1$ is faithful by Lemma~\ref{l-faith}, because
$FH_1$ is a Frobenius group and the action of $F$ is nontrivial
since $C_U(F)=0$ by property~(W3). Switching to multiplicative
notation also for the additive group of $U$, we now have the
semidirect product $G_1=UQ$ admitting the Frobenius group of
automorphisms $(F/C_F(G_1))H_1$ such that $C_{U}(F/C_F(G_1))=1$.
Consider $\bar G_1=G_1/O_q(G_1)$, keeping the same notation for
$U$, $F$, $H_1$. Note that $\bar c\ne 1$, that is, $c\not\in O_q(G_1)$, as $c$ is
nontrivial on $U$. Then the hypotheses of the proposition hold for
$\bar G_1$ and $(F/C_F(G_1))H_1$. We claim that $\bar c\in F(C_{\bar G_1}(H_1))$; indeed,
$C_{\bar G_1}(H_1)=C_{U}(H_1)C_{Q}(H_1)$ is a
$\{p,q\}$-group, in which $C_{U}(H_1)$ is a normal $p$-sub\-group
centralized by the $q$-ele\-ment $\bar c$ by
Lemma~\ref{c-triv-cu}. If $H_1\ne H$, then by
induction on $|H|$ we must have $\bar c\in U$, a contradiction.

Thus, it remains to consider the case where $H_1=H$, that is,
$W_1$ is $H$-inva\-ri\-ant, which is assumed in what follows.

We now focus on the action on $W_1$, using bars to denote the
images of $Q$ and its elements in their action on~$W_1$. We can regard $H$
as acting by automorphisms on $\overline Q$.
Let $\zeta _2(\overline Q)$ be the second centre of $\overline Q$. We obviously have
$
[[H,\, \overline c],\, \zeta _2(\overline Q)]=[1,\, \zeta _2(\overline Q)]=1.
$
We also have
$
[[\overline c,\, \zeta _2(\overline Q)],\, H]\leq [Z(\overline Q),\, H]=1,
$
since $Z(\overline Q)$ is represented on the homogeneous $k\overline Q$-module $W_1$ by scalar linear transformations.
Therefore by the Three Subgroup Lemma,
\begin{equation}\label{z2h}
[[\zeta _2(\overline Q),\, H],\, \overline c]=1.
\end{equation}

By the choice of $c\in Z(C_G(H))$ we also have
\begin{equation}\label{cz2}
[C_{\zeta _2(\overline Q)}(H),\, \overline c]=1,
\end{equation}
because $C_Q(H)$ covers $C_{\overline Q}(H)$. Since
the action of $H$ on $Q$ is coprime by hypothesis, we have
 $
 \zeta _2(\overline Q)=[\zeta _2(\overline Q), H]C_{\zeta _2(\overline Q)}(H).
 $
 Therefore equalities
 \eqref{z2h} and \eqref{cz2} together imply the equality
\begin{equation}\label{z2}
[\overline c ,\,\zeta _2(\overline Q)]=1.
\end{equation}

Now let $F_1$ denote the stabilizer of $W_1$ in $F$, so that
the stabilizer of $W_1$ in $FH$ is equal to $F_1H$. Then for any
element $f\in F\setminus F_1$ the component $W_1f$ is outside the
$H$-orbit of $W_1$, which is equal to $\{ W_1\}$ in the case under consideration.
As mentioned above, $c$ acts trivially on all the Wedderburn components outside
the $H$-orbit of $W_1$. Thus, $c$ acts trivially on $W_1f$, which is equivalent
to $c^{f^{-1}}$ acting trivially on $W_1$. In other words, $\overline{c^{x}}=1$ in the action on $W_1$ for any
$x\in F\setminus F_1$. (Note that it does not matter that $W_1$ is
not $F$-invariant: for any $g\in F$ the element $c^g$ belongs to $Q$, which acts on $W_1$.)
Since $Q=\langle c^F\rangle$ by \eqref{zam-c},
we obtain that $\overline
Q=\langle \overline c^{F_1}\rangle$. In view of the $F_1$-invariance of the section $\overline Q$
we can apply conjugation by any $g\in F_1$ to equation \eqref{z2} to obtain that
$$
[\overline c^{g},\, \zeta
_2(\overline Q)^{g}]=[\overline c^{g},\, \zeta _2(\overline Q)]=1.
$$
As a result,
$$
[\overline Q,\, \zeta _2(\overline Q)]=[\langle \overline c^{F_1}\rangle,\, \zeta
_2(\overline Q)]=1.
$$
This means that $\overline Q$ is abelian.

In the case of $\overline Q$ abelian we arrive at a contradiction similarly to how this was done
in \cite[Theorem~2.7(c)]{khu-ma-shu}.
The sum of the $W_i$ over all regular $H$-orbits is obviously a
free $kH$-module. Since the whole $W$ is also
a free $kH$-module by property~(W4), the component $W_1$ must also be a free $kH$-module, as
a complement of the sum over all regular $H$-orbits. Since $\overline Q$ is abelian, $c$ acts on $W_1$ by
a scalar linear transformation. By Lemma~\ref{c-triv-free} (or simply because
$c$ has fixed points in $W_1$, as $C_{W_1}(H)\ne 0$) it follows that, in fact, $c$ must act
trivially on $W_1$, a contradiction with property~(W1). Proposition~\ref{p1} is proved.
\ep

\begin{proposition}\label{p2}
Suppose that a soluble finite group $G$ admits a Frobenius group of
automorphisms $FH$ of coprime order with kernel $F$ and complement $H$ such that
$V=F(G)=O_p(G)$ is an elementary abelian $p$-group and
$C_G(H)$ is nilpotent. If $C_V(F)=1$, then $G=VC_G(F)$.
\end{proposition}

\bp
The quotient $\bar G =G/V$ acts faithfully on $V$. We claim that $F$ acts trivially on
$\bar G$. Suppose not; then $F$ acts non-trivially on the Fitting subgroup
$F(\bar G)$. Indeed, otherwise $[\bar G ,F]$ acts trivially on $F(\bar G)$, which
contains its centralizer, so then $[\bar G ,F]\leq F(\bar
G)$ and $F$ acts trivially on $\bar G$ since the action is coprime.
Thus, $F$ acts non-trivially on some Sylow $q$-subgroup $Q$ of $F(\bar G)$. But then
there is a nontrivial fixed point of $H$ in $Q$ by Lemma~\ref{l-free}. This would contradict
 Proposition~\ref{p1}, since here $C_G(H)$ is nilpotent. \ep

\begin{proof}[Proof of Theorem \ref{t-n}]
Recall that $G$ is a soluble finite group  admitting 
a Frobenius group of automorphisms $FH$ of coprime order with kernel $F$ and complement $H$ such that
$|C_G(F)|=m$ and $C_G(H)$ is nilpotent. We claim that $|G/F(G)|$ is $(m,|F|)$-bounded.

We can assume that $G=[G,F]$. For every
$p\nmid |C_G(F)|$ we have $G=O_{p',p}(G)$ by Proposition~\ref{p2} applied to the quotient
of $G$ by the pre-image of the Frattini subgroup of $O_{p',p}(G)/O_{p'}(G)$ with $V$ equal to the
Frattini quotient of $O_{p',p}(G)/O_{p'}(G)$. It remains to prove that
$|G/O_{p',p}(G)|$ is $(m,|F|)$-bounded for every $p$, since then
the index of $F(G)=\bigcap O_{p',p}(G)$ will be at most the
product of these bounded indices over the bounded set of primes
$p$ dividing $ |C_G(F)|$.

The quotient $\bar G =G/O_{p',p}(G)$ acts faithfully on the Frattini quotient $X$
of $O_{p',p}(G)/O_{p'}(G)$. It is sufficient
to bound the order of the Fitting subgroup $F(\bar G)$. Thus, we can assume that $\bar
G=F(\bar G)$. Therefore $\bar
G$ is a $p'$-group, so that the order of $\bar
G FH$ is coprime to $p$, the characteristic of the ground field of $X$ regarded as a vector space over ${\Bbb F}_p$.

Let $X=Y_1\oplus \cdots \oplus Y_s\oplus Z_1\oplus\dots\oplus Z_t$
be the decomposition of $X$ into the direct sum of irreducible
${\Bbb F}_p\bar G F$-submodules, where $C_{Y_i}(F)\ne 0$ for all $i$
and $C_{Z_j}(F)= 0$ for all $j$. (Here either of $s$ or $t$ can be
zero.) By the Hartley--Isaacs theorem \cite[Theorem~B]{ha-is} there is an $|F|$-bounded number $\al (F)$ such that
$\dim Y_i\leq \al (F)\dim
C_{Y_i}(F)$ for every $i$; here $\dim
C_{Y_i}(F)$ is the multiplicity of the trivial ${\Bbb F}_pF$-submodule, which does appear in $Y_i$ by definition.
Therefore the order of $Y_1\oplus \cdots \oplus
Y_s$ is $(m,|F|)$-bounded.

Let $Y= Y_1\oplus \cdots \oplus Y_s$. We claim that $Y$ is
$H$-invariant. Indeed, $C_{Y_i}(F)\ne 0$ for every $i$. Hence,
$C_{Y_ih}(F)=C_{Y_ih}(F^h)=C_{Y_i}(F)h\ne 0$ for every $h\in H$.
Since $Y_ih$ is also an irreducible ${\Bbb F}_p\bar G
F$-submodule, we must have either $Y_ih\subseteq Y$ or $Y_ih\cap
Y=0$. But the second possibility is impossible, since $C_{Y_ih}(F)\ne 0$ and $C_X(F)\subseteq Y$ by construction.

The subgroup $C_{G}(Y)$ is normal and $FH$-invariant. Since $|Y|$ is also $(m,|F|)$-bounded,
$|G/C_G(Y)|$ is also $(m,|F|)$-bounded.

By Maschke's theorem, $X=Y\oplus Z$, where $Z$ is normal and
$FH$-invariant. Obviously, $C_Z(F)=0$.

We now revert to the multiplicative notation and regard $X,Y,Z$ as sections of the group $G$.
 The group $C_G(Y)/X$ acts on
$Z$. Since $C_G(Y)/X$ is faithful on $X$ and the
action is coprime, it is also faithful on $Z$, so that $Z=O_p(G_1)$, where $G_1$ is the
semidirect product $G_1=Z\rtimes C_G(Y)/X$.
The group $G_1$ with the induced action of $FH$ satisfies the hypotheses of
Proposition~\ref{p2} (with $V=Z$). Indeed, $F$ acts fixed-point-freely on
$Z=O_p(G_1)$, and $C_{G_1}(H)$ is covered by the images of subgroups of $C_G(H)$ and therefore is also nilpotent.
By Proposition~\ref{p2} we obtain that $F$ acts
trivially on $C_G(Y)/X$, so that $|C_G(Y)/X| \leq |C_G(F)|=m$. As
a result, $|G/O_{p',p}(G)|=|G/C_G(Y)|\cdot |C_G(Y)/X|$ is $(m,|F|)$-bounded, as required.
 \ep

\section{Lie ring theorem}\label{s-l}

In this section a finite Frobenius group
$FH$
with cyclic kernel $F$ of order $n$ and
complement $H$ of order $q$ acts by automorphisms on a
Lie ring $L$ in whose ground ring $n$ is invertible.
If the fixed-point subring $C_L(H)$ of the
complement is nilpotent of class $c$ and $C_L(F)=0$,
that is, the kernel $F$ acts without non-trivial fixed points on
$L$, then by the Makarenko--Khukhro--Shumyatsky
theorem~\cite{khu-ma-shu} the Lie ring $L$ is nilpotent of
$(c,q)$-bounded class. In Theorem~\ref{t-l} we have
$|C_L(F)|=m$ and need to prove that $L$ contains a nilpotent
subring of $(m,n,c)$-bounded index (in the additive group)
and of $(c,q)$-bounded nilpotency class.

The proof of Theorem~\ref{t-l} uses the method of graded centralizers, which
was developed in the authors' papers on groups and Lie rings with almost regular automorphisms
\cite{kh2,khmk1,khmk3, khmk4,khmk5}; see also Ch.~4 in \cite{kh4}. This method consists in
 the following. In the proof of Theorem~\ref{t-l} we can assume that
 the ground ring contains a primitive $n$th root of unity~$\omega$. Let $F=\langle \varphi\rangle$.
Since $n$ is invertible in the ground ring, then $L$ decomposes into the direct sum
of the ``eigenspaces''
$L_j=\{ a\in L\mid a^{\varphi}=\omega ^ja\}$, which are also
 components of a $(\Z /n\Z)$-grading: $[L_s,\, L_t]\subseteq
L_{s+t},$ where $s+t$ is calculated modulo $n$. In each of the $L_i$,
$i\ne 0$, certain additive subgroups
 $L_i(k)$ of bounded index --- ``graded centralizers'' ---
 of increasing levels $k$ are successively constructed, and simultaneously
certain elements (representatives) $x_i(k)$ are fixed, all this up
to a certain $(c,q)$-bounded level $T$. Elements of $L_j(k)$
have a centralizer property with respect to the fixed elements of
lower levels: if a commutator (of bounded weight) that involves
exactly one element $y_j(k)\in L_j(k)$ of level $k$ and some fixed
elements $x_i(s)\in L_i(s)$ of lower levels $s<k$ belongs to
$L_0$, then this commutator is equal to~$0$. The sought-for
subring $Z$ is generated by all the $L_i(T)$, $i\ne 0$, of the
highest level $T$. The proof of the fact that the subring $Z$ is
nilpotent of bounded class is based on a combinatorial fact
following from the Makarenko--Khukhro--Shumyatsky
theorem~\cite{khu-ma-shu, khu-ma-shu-DAN} for the case $C_L(F)=0$
(similarly to how combinatorial forms of the Higman--Kreknin--Kostrikin theorems were used in
our papers \cite{kh2,mk05}
on almost fixed-point-free automorphisms). The
question of nilpotency is reduced to consideration of commutators of a
special form, to which the aforementioned centralizer
property is applied.

First we recall some definitions and notions.
Products in a Lie ring are called ``commutators''.
The Lie subring generated by a subset~$S$ is denoted by $\langle S\rangle $, and the ideal by
${}_{{\rm id}}\!\left< S \right>$.

Terms of the lower central series of a Lie ring $L$
are defined by induction: $\gamma_1(L)=L$; $\gamma_{i+1}(L)=[\gamma_i(L),L].$
By definition a Lie ring $L$ is nilpotent of class~$h$ if
$\gamma_{h+1}(L)=0$.

A simple commutator $[a_1,a_2,\dots ,a_s]$ of weight (length)
 $s$ is by definition the commutator $[\dots [[a_1,a_2],a_3],\dots ,a_s]$.
 By the Jacobi identity $[a,[b,c]]=[a,b,c]-[a,c,b]$ any (complex, repeated)
commutator in some elements in any Lie ring can be
expressed as a linear combination of simple commutators of
the same weight in the same elements. Using also the anticommutativity
$[a,b]=-[b,a]$, one can make sure that in this linear combination
all simple commutators begin with some pre-assigned
element occurring in the original commutator. In particular, if
$ L=\langle S\rangle $, then the additive group $L$ is generated by simple
commutators in elements of~$S$.

Let $A$ be an additively written abelian group. A Lie ring $L$
is \textit{$A$-graded} if
$$L=\bigoplus_{a\in A}L_a\qquad \text{ and }\qquad[L_a,L_b]\subseteq L_{a+b},\quad a,b\in A,$$
where the grading components $L_a$ are additive subgroups of $L$. Elements of the $L_a$
are called \textit{homogeneous} (with respect to this grading), and commutators in homogeneous
elements \textit{homogeneous commutators}. An additive subgroup
 $H$ of $L$ is said to be \textit{homogeneous}
if $H=\bigoplus_a (H\cap L_a)$; then we set $H_a=H\cap L_a$.
Obviously, any subring or an ideal generated by homogeneous
additive subgroups is
 homogeneous. A homogeneous subring and the
quotient ring by a homogeneous ideal can be regarded as
$A$-graded rings with induced grading.

\begin{proof}[Proof of Theorem~\ref{t-l}] Recall that
$L$ is a Lie ring admitting a
Frobenius group of automorphisms $FH$ with cyclic kernel
$F=\langle \varphi\rangle$ of order $n$ invertible in the ground ring of $L$
 and with complement $H$ of order $q$ such that the subring $C_L(H)$ of fixed points
of the complement is nilpotent of class $c$ and the subring
of fixed points $C_L(F)$ of the kernel has order $m=|C_L(F)|$.
Let $\omega$ be a primitive $n$th root of unity. We extend the
ground ring by $\omega$ and denote by $\widetilde L$ the ring
$L\otimes _{{\Bbb Z} }{\Bbb Z} [\omega ]$. The group $FH$
naturally acts on $\widetilde L$; then
$C_{\widetilde L}(H)$ is nilpotent of class~$c$, and
$|C_{\widetilde L}(F)|\leq |C_L(\varphi
)|^n=m^n$. If $\tilde L$ has a nilpotent subring of
$(m,n,c)$-bounded index in the additive group and of
$(c,q)$-bounded nilpotency class, then the same holds for $L$.
Therefore we can replace $L$ by $\tilde L$, so that henceforth we
assume that the ground ring contains $\w$.

\begin{definition} We define $\varphi$-\textit{components} $L_k$ for
$k=0,\,1,\,\ldots ,n-1$ as the ``eigensubspaces''
$$L_k=\left\{ a\in L\mid a^{\varphi}=\w ^{k}a\right\} .$$
\end{definition}

Since $n$ is invertible in the ground ring, we have $L= L_0 \oplus L_1\oplus \dots \oplus L_{n-1}$
(see, for example,~\cite[Ch.~10]{hpbl}). This decomposition
is a $({\Bbb Z}/n{\Bbb Z})$-grading due to the obvious
inclusions $[L_s,\, L_t]\subseteq L_{s+t\,({\rm mod}\,n)}$, so that the $\varphi$-components are
the grading components.

\begin{definition}
We refer to elements, commutators, additive subgroups that are homogeneous with respect to this
grading into $\varphi$-components as being \textit{$\varphi$-homogeneous}.
\end{definition}

\begin{Index Convention} Henceforth a small letter
with index $i$ denotes an element of the $\varphi$-component $L_i$, so that the index only indicates the
$\varphi$-component to which this
element belongs: $x_i\in L_i$. To lighten the notation we will not
use numbering indices for elements in $L_j$, so that
different elements can be denoted by the same symbol when
it only matters to which $\varphi$-component these elements belong. For example, $x_1$ and $x_1$ can be
different elements of $L_1$, so that $[x_1,\, x_1]$ can be a nonzero element of
$L_2$. These indices will be considered modulo~$n$; for example, $a_{-i}\in
L_{-i}=L_{n-i}$.
\end{Index Convention}

Note that under the Index Convention
a $\varphi$-homogeneous commutator
 belongs to the $\varphi$-component $L_s$, where
 $s$ is the sum modulo $n$ of the indices of all the elements occurring in this
commutator.

Since the kernel $F$ of the Frobenius group $FH$ is cyclic,
the complement $H$ is also cyclic. Let $H= \langle h
\rangle$ and $\varphi^{h^{-1}} = \varphi^{r}$ for some $1\leq
r \leq n-1$. Then $r$ is a primitive $q$th root of unity in the ring ${\Bbb Z}/n {\Bbb Z}$.

The group $H$ permutes the $\varphi$-components $L_i$ as follows:
$L_i^h = L_{ri}$ for all $i\in \Bbb Z/n\Bbb Z$.
Indeed, if $x_i\in L_i$, then $(x_i^{h})^{\varphi} =
x_i^{h\varphi h^{-1}h} = (x_i^{\varphi^{r}})^h =\omega^{ir}x_i^h$.

\begin{notation} In what follows, for a given $u_k\in L_k$ we denote the element
$u_k^{h^i}$ by $u_{r^ik}$ under the Index Convention,
since $L_k^{h^i} = L_{r^ik}$. We denote the $H$-orbit of an element $x_i$ by
$O(x_{i})=\{x_{i},\,\, x_{ri},\dots, x_{r^{q-1}i}\}$.
\end{notation}

\paragraph{Combinatorial theorem.} We now prove a combinatorial consequence of
the Makarenko--Khukhro--Shumyatsky theorem in~\cite{khu-ma-shu}. Recall
that we use the centralizer notation $C_L(A)$ for the fixed-point
subring of a Lie ring $L$ admitting a group of automorphisms~$A$.

\begin{theorem} [{\cite[Theorem~5.6(iii)]{khu-ma-shu}}]\label{kh-ma-shu10-1}
Let $FH$ be a Frobenius group with cyclic kernel $F$
of order $n$ and complement $H$ of order $q$. Suppose that $FH$
acts by automorphisms on a Lie ring
$M$ in the ground ring of which $n$ is invertible. If
$C_M(F)=0$ and $C_M(H)$ is nilpotent
of class $c$, then for some $(c,q)$-bounded number $f=f(c,q)$ the
Lie ring $M$ is nilpotent of class at most $f$.
\end{theorem}

We need the following consequence for our Lie ring $L$ under the
hypotheses of Theorem~\ref{t-l}, the ground ring of which we
already assume to contain the $n$th primitive root of unity, so
that $L=L_0\oplus L_1\oplus\dots\oplus L_{n-1}$.

\begin{proposition}\label{combinatorial}
Let $f(c,q)$ be the function in Theorem~\ref{kh-ma-shu10-1} and let $T=f(c,q)+1$.
Every simple
$\varphi$-homogeneous
commutator $[x_{i_1}, x_{i_2},\ldots, x_{i_{T}}]$ of weight $T$ with non-zero indices can be
represented as a linear combination of $\varphi$-homogeneous simple
commutators of the same weight $T$ in elements of the union of $H$-orbits
$\bigcup_{s=1}^T O(x_{i_s})$ each of which contains an initial segment
with zero sum of indices modulo $n$
and includes exactly the same number of elements of each $H$-orbit
$O(x_{i_s})$ as the original commutator.
\end{proposition}

\begin{proof}
The idea of the proof is application of
Theorem~\ref{kh-ma-shu10-1} to a free Lie ring with operators
$FH$. Recall that we can assume that the ground ring $R$ of our
Lie ring $L$ contains a primitive root of unity $\w$ and a
multiplicative inverse of $n$.  Given arbitrary (not necessarily
distinct) non-zero
 elements $i_1, i_2,\cdots, i_T\in \Bbb Z /n{\Bbb Z}$, we
consider a free Lie ring $K$ over $R$ with $qT$ free generators in
the set
$$Y=\{\underbrace{y_{i_1}, y_{ri_1}, \ldots, y_{r^{q-1}i_1}}_{O(y_{i_1})},\,\,\,
\underbrace{y_{i_2}, y_{ri_2}, \ldots,
y_{r^{q-1}i_2}}_{O(y_{i_2})},\ldots,
\underbrace{y_{i_{T}},y_{ri_T}, \ldots,
y_{r^{q-1}i_T}}_{O(y_{i_T})}\},$$
where indices  are formally assigned and regarded modulo $n$ and the subsets
$O(y_{i_s})=\{y_{i_s}, y_{ri_s}, \ldots, y_{r^{q-1}i_s}\}$ are disjoint.
Here, as in the Index Convention, we do not use numbering indices, that is,
all elements $y_{r^ki_j}$ are by definition different free generators, even if
indices coincide. (The Index Convention will come into force in a moment.) For
every $i=0,\,1,\,\dots ,n-1$ we define
the additive subgroup $K_i$ generated by all
commutators in the generators $y_{j_s}$ in which the sum of
indices of all entries is equal to $i$ modulo $n$. Then
$K=K_0\oplus K_1\oplus \cdots \oplus K_{n-1}$. It is also obvious that
$ [K_i,K_j]\subseteq K_{i+j\,({\rm mod\, n)}}$; therefore this is a
$({\Bbb Z} /n{\Bbb Z})$-grading. The Lie ring $K$ also has
the natural ${\Bbb N}$-grading $K=G_1(Y)\oplus G_2(Y)\oplus \cdots $
with respect to the generating set $Y$, where $G_i(Y)$ is the additive subgroup generated by all
commutators of weight $i$ in elements of $Y$.

We define an action of the Frobenius group $FH$ on $K$ by setting
$k_i^{\varphi}=\omega^i k_i$ for $k_i\in K_i$ and extending this
action to $K$ by linearity. Since $K$ is the direct sum
of the additive subgroups $K_i$ and $n$ is invertible in the ground ring, we have
$K_i=\{k\in K \mid k^{\varphi}=\omega^i k \}$. An action of $H$ is defined on the generating set $Y$
as a cyclic permutation of elements in each subset $O(y_{i_s})$ by the rule
$(y_{r^ki_s})^h=y_{r^{k+1}i_s}$ for $k=0,\ldots, q-2$ and $(y_{r^{q-1}i_s})^h=y_{i_s}$.
Then $O(y_{i_s})$ becomes the $H$-orbit
of an element $y_{i_s}$. Clearly, $H$ permutes the components $K_i$ by the rule
$K_i^h = K_{ri}$ for all $i\in \Bbb Z/n\Bbb Z$.

Let $J={}_{{\rm id}}\!\left< K_0 \right>$ be the ideal
generated by the $\varphi$-component $K_0$. Clearly,
the ideal $J$ consists of linear combinations of commutators
in elements of $Y$ each of which contains a subcommutator with
zero sum of indices modulo $n$. The
ideal $J$
is generated by homogeneous elements with respect to the gradings
$K=\bigoplus_i G_i(Y)$ and $K=\bigoplus_{i=0}^{n-1} K_i$ and therefore
is homogeneous with respect to both gradings.
Note also that the ideal $J$ is obviously $FH$-invariant.

Let $I={}_{{\rm id}}\!\left< \gamma_{c+1}(C_K(H)) \right>^F$
be the smallest $F$-invariant ideal containing the subring
$\gamma_{c+1}(C_K(H))$. The ideal $I$ is obviously homogeneous with respect to the grading
$K=\bigoplus_i G_i(Y)$ and is $FH$-invariant. Being $F$-invariant, the ideal $I$ is also
homogeneous with respect to the grading $K=\bigoplus_{i=0}^{n-1} K_i$.
Indeed, for $z\in I$, let $z=k_0+k_1+\cdots+k_{n-1}$, where $k_i\in K_i$.
On the other hand, for every $i=0,\ldots, n-1$ we have
$z_i:= \sum_{s=0}^{n-1} \omega^{-is}z^{\varphi^s}\in K_i$ and $nz=\sum_{j=0}^{n-1}z_i$.
Since $n$ is invertible in the ground ring,
$z=1/n\sum_{j=0}^{n-1}z_i$. Hence $k_i=(1/n) z_i=(1/n)\sum_{s=0}^{n-1}
\omega^{-is}z^{\varphi^s}$. Since $I$ is $\varphi$-invariant, $z^{\varphi^{j}}\in I$ for all $j$;
therefore, $k_i\in I$ for all $i$, as required.

Consider the quotient Lie ring $M=K/(J+I)$. Since the ideals $J$ and $I$
are homogeneous with respect to the gradings $K=\bigoplus_i G_i(Y)$ and
$K=\bigoplus_{i=0}^{n-1} K_i$, the quotient ring $M$ has
the corresponding induced gradings. Furthermore, $M_0=0$ by construction of $J$. Therefore
in the induced action of the group $FH$ on $M$ we have $C_M(F)=0$, since $n$ is invertible in the ground ring.
The group $H$ permutes the grading components of $M=M_1\oplus\dots \oplus M_{n-1}$ with regular orbits of length $q$.
Therefore elements of $C_M(H)$ have the form $m+m^h+\dots +m^{h^{q-1}}$. Hence $C_M(H)$ is the image of $C_K(H)$
in $M=K/(I+J)$ and $\gamma_{c+1}(C_M(H)) =0$ by construction of $I$.
By Theorem~\ref{kh-ma-shu10-1} $M$ is nilpotent of $(c,q)$-bounded class
$f=f(c,q)$. Consequently,
$$
[y_{i_1}, y_{i_2},\ldots, y_{i_{T}} ]\in J+I={}_{{\rm id}}\!\left<
K_0 \right>+{}_{{\rm id}}\!\left<\gamma_{c+1}(C_K(H)) \right>^F.
$$
Since both ideals are homogeneous with respect to the grading $K=\bigoplus_i G_i(Y)$,
this means that the commutator $[y_{i_1}, y_{i_2},\ldots, y_{i_{T}}]$
is equal modulo the ideal $I$ to a linear combination
of commutators of the same weight $T$
 in elements of $Y$ each of which
contains a subcommutator with zero sum of indices modulo $n$.

We claim that in addition every commutator in this linear combination
can be assumed to include exactly one element of the $H$-orbit $O(y_{i_s})$ for every $s=1,\ldots, T$.
 For every $s=1,\ldots,T$ we consider the homomorphism $\theta_s$
extending the mapping
 $$
 O(y_{i_s}) \rightarrow 0; \qquad y_{i_k}\rightarrow y_{i_k} \quad \text{if}\quad k\neq s.
 $$
 We say for brevity that a commutator \textit{depends on $ O(y_{i_s})$} if it involves at least one element of
 $ O(y_{i_s})$. The homomorphism $\theta_s$ sends to 0 every commutator in
elements of $Y$ that depends on $ O(y_{i_s})$, and acts as identity on commutators
that are independent of $ O(y_{i_s})$. Hence $\theta_s$ clearly
commutes with the action of $H$. The automorphism $\varphi$ acts on any
homogeneous commutator by multiplication by some power of $\w$.
Therefore $\theta_s$ also commutes with the action of $F$. It
follows that
the ideal $I$ is invariant under $\theta_s$. Indeed,
$C_K(H)$ is
$\theta_s$-invariant, then so is the ideal ${}_{{\rm id}}\!\left< \gamma_{c+1}(C_K(H)) \right>$.
Since $\theta_s$ commutes with the action of $F$, the $F$-closure of the latter ideal,
which is $I$, is also $\theta_s$-invariant.

We apply the homomorphism $\theta_s$ to the commutator $[y_{i_1},
y_{i_2},\ldots, y_{i_{T}}]$ and its representation modulo $I$ as a
linear combination of commutators in elements of $Y$ of weight $T$
containing subcommutators with zero sum of indices modulo $n$.
The image $\theta_s([y_{i_1}, y_{i_2},\ldots,
y_{i_{T}}])$ is equal to $0$, as well as the image of any
commutator depending on $O(y_{i_s})$. Hence in the representation of
$[y_{i_1}, y_{i_2},\ldots, y_{i_{T}}]$ the part of the linear combination
in which commutators are independent of $O(y_{i_s})$ is equal to zero and can be excluded
from the expression. By applying
consecutively $\theta_s$, $s=1,\ldots, T$, and excluding
commutators independent of $O(y_{i_s})$, $s=1,\ldots,
T$, in the end we obtain modulo $I$ a linear combination of
commutators each of which contains at least one element from every
orbit $O(y_{i_s})$, $s=1,\ldots, T$. Since under these
transformations the weight of commutators remains the same and is
equal to $T$, no other elements can appear, and every commutator
will contain exactly one element in every orbit $O(y_{i_s})$,
$s=1,\ldots, T$.

Now suppose that $L$ is an arbitrary Lie ring over $R$ satisfying
the hypothesis of Proposition~\ref{combinatorial}. Let $x_{i_1},
x_{i_2},\ldots, x_{i_{T}} $ be arbitrary $\varphi$-homogeneous
elements of $L$.
We define the homomorphism $\delta$ from the free Lie ring $K$
into $L$ extending the mapping
$$
y_{r^ki_s}\to x_{i_s}^{h^k}\quad \text{for} \quad
s=1,\ldots,T \quad \text{and}\quad k=0,1,\ldots,q-1.
$$
It is easy to see that $\delta$ commutes with the action of $FH$ on $K$ and $L$. Therefore
$\delta (O(y_{i_s}))=O(x_{i_s})$ and
$\delta (I)=0$, since $\gamma_{c+1}(C_L(H)) =0$ and $\delta(C_K(H))\subseteq C_L(H)$.
We now apply $\delta$ to the representation
 of the commutator $[y_{i_1}, y_{i_2},\ldots, y_{i_{T}}]$ constructed above. Since $\delta(I)=0$,
 as the image we
 obtain a representation of the commutator $[x_{i_1},x_{i_2},\ldots, x_{i_{T}} ]$ as a linear
combination of commutators of weight $T$
in elements of the set $\delta (Y)=\bigcup_{s=1}^T O(x_{i_s})$
each of which contains exactly the same number of elements from
every $H$-orbit $O(x_{i_s})$, $s=1,\ldots, T$, as the original
commutator, and has a subcommutator with zero sum of indices
modulo $n$. (Unlike the disjoint orbits $O(y_{i_s})$ in $K$, their
images may coincide in $L$, which is why we can only claim ``the
same number of elements'' from each of them, rather than exactly
one from each.)

Finally,
by the anticommutativity and Jacobi identities we can transform this  linear combination
into another one of simple commutators in the same elements each having
an initial segment with zero sum of indices.
The proposition is proved.
\end{proof}

\begin{definition} We define a \textit{KMS-transformation} of a commutator $[x_{i_1},
x_{i_2},\ldots, x_{i_l}]$ for $l\geq T$ to be its representation according to Proposition~\ref{combinatorial} as
a linear combination of simple
commutators of the same weight
in elements of $\bigcup_{s=1}^T O(x_{i_s})$
each of which
 has an initial segment from $L_0$ of weight $\leq
T$, that is, commutators of the form
$$
[c_0,y_{j_{w+1}},\ldots, y_{j_{l}}],
$$
where $c_0=[y_{j_1},\ldots, y_{j_w}]\in L_0$, $w\leq T$, $y_{j_k}\in O(x_{i_1})\cup \cdots \cup O(x_{i_l})$,
with subsequent re-denoting
\be\label{re-den}
z_{j_{w+1}}=-[c_0, y_{j_{w+1}}],\qquad  z_{i_{w+s}}=y_{i_{w+s}}
\,\,\,\text{for}\,\,\, s>1.
\ee
\end{definition}

The following assertion is obtained by repeated application
of KMS-transformations.

\begin{proposition} \label{kh-ma-shu-transformation}
For any positive integers $t_1$ and $t_2$ there exists a $(t_1,t_2, c,q)$-bounded positive
integer $V=V(t_1,t_2, c,q)$ such that any simple $\varphi$-homogeneous commutator
$[x_{i_1}, x_{i_2},\ldots, x_{i_{V}}]$ of weight $V$ with non-zero indices
can be
represented as a linear combination of $\varphi$-homogeneous
commutators of the same weight
in elements of the set $X=\bigcup_{s=1}^V
O(x_{i_s})$
each having either a subcommutator of the form
\begin{equation}\label{f1}
[u_{k_1},\ldots, u_{k_s}]
\end{equation}
with $t_1$ different initial segments with zero sum
of indices modulo~$n$, that is, with $k_1+k_2+\cdots+k_{r_i}\equiv 0\; ({\rm mod}\, n)$ for
$ 1<r_1<r_2<\cdots<r_{t_1}=s$,
or a subcommutator of the form
\begin{equation}\label{f2}
[u_{k_0}, c^1,\ldots, c^{t_2}],
\end{equation}
where $u_{k_0}\in X$, every $c^i$ belongs to $L_0$ (with numbering upper
indices $i=1,\ldots,t_2$)
and has the form $[a_{k_1},\ldots, a_{k_i}]$ for $a_{k_j}\in X$ with $k_1+\cdots+k_i\equiv 0\; ({\rm mod}\, n)$.
Here we can set
$V(t_1,t_2,c,q)= \sum_{i=1}^{t_1}((f(c,q)+1)^2t_2)^i+1$.
\end{proposition}

\begin{proof} The proof practically word-for-word repeats the proof
of the Proposition in~\cite{kh2} (see also Proposition~4.4.2 in
\cite{kh4}), with the HKK-transformations replaced by the
KMS-transformations. Namely, the KMS-transformation is
applied to the initial segment of length $T$ of the commutator.
Each of the resulting commutators contains a subcommutator in
$L_0$, which becomes a part of a new element of type $z_1$ in
\eqref{re-den}. Then the KMS-transformation is applied again, and
so on. Note that images under $H$ of  commutators in elements of the $H$-orbits
of the $x_{i_j}$ are again commutators in elements of the $H$-orbits of the $x_{i_j}$.
Subcommutators in $L_0$ are thus accumulated, either
nested (for example, if, say, at the second step the element $z_1$
containing $c_0\in L_0$ is included in the new subcommutator in
$L_0$), or disjoint (if, say, at the second step $z_1$ is not
included in a new subcommutator in $L_0$). Nested accumulation leads to \eqref{f1}, and disjoint to \eqref{f2}.
The total number of subcommutators in $L_0$ increases at every step, and
the required linear combination is achieved after sufficiently many steps.

The precise details of the proof by induction can be found in~\cite{kh2} or \cite[\S\,4.4]{kh4}.
The only difference is that HKK-transformations always produce commutators in the same elements as the original
commutator $[x_{i_1}, x_{i_2},\ldots, x_{i_{V}}]$, while after the
KMS-transformation the resulting commutators may also involve
images of the elements $x_{i_j}$ under the action of $H$.
This is why elements of the $H$-orbits have to
appear in the conclusion of the proposition.
\end{proof}

\paragraph{Representatives and graded centralizers.}
We begin construction of graded centralizers by induction on the level
taking integer values from $0$ to $T$, where
the number $T=T(c,q)= f(c,q)+1$ is defined in
Proposition~\ref{combinatorial}. A graded centralizer $L_j(s)$
of level $s$ is a certain additive subgroup
 of the $\varphi$-component~$L_j$. Simultaneously with construction of
graded centralizers we fix certain elements of them ---
representatives of various levels --- the total number of which
is $(m,n,c)$-bounded.

\begin{definition} The \textit{pattern} of a commutator in $\varphi$-homogeneous elements
(of various~$L_i$) is defined as its bracket structure together with the arrangement of
indices under the Index Convention. The \textit{weight} of a pattern is the
weight of the commutator. The commutator itself is called the value of its
pattern on given elements.
\end{definition}

\begin{definition} Let $\vec x=(x_{i_1},\dots ,x_{i_k})$ be
an ordered tuple of elements $x_{i_s}\in L_{i_s}$,
$i_s\ne 0$, such that $i_1+\dots + i_k\not\equiv 0\,
({\rm mod}\, n)$. We set $j=-i_1-\dots - i_k\,({\rm mod}\, n)$ and
define the mapping
\begin{equation}\label{vartheta}
\vartheta _{\vec x}: y_j\rightarrow [y_j, x_{i_1}, \dots , x_{i_k}].
\end{equation}
\end{definition}
By linearity this is a homomorphism of the additive subgroup
 $L_j$ into $L_0$. Since $|L_0|= m$, we have $|L_j : {\rm Ker}\,
\vartheta_{\vec x}|\leq m$.

\begin{notation} Let $U=U(c,q)=V(T, T-1, c,q )$, where
$V$ is the function in Proposition~\ref{kh-ma-shu-transformation}. \end{notation}

\begin{definition0} Here we only fix
representatives of level $0$. First, for every pair $({\bf p},c)$ consisting of a pattern ${\bf p}$
of a simple commutator of weight $\leq U$ with non-zero indices of entries and zero
sum of indices and a commutator $c\in L_0$ equal to the value of ${\bf p}$ on
$\varphi$-homogeneous elements of various $L_{i}$, $i\ne 0$, we fix one such representation.
The elements of $L_{j}$,\, $j\ne 0$, occurring in
this fixed representation of $c$ are called
\textit{representatives of level~$0$}. Representatives of level $0$
are denoted by $x_j (0)$  (with letter ``$x$'')  under the Index Convention
(recall that the same symbol can denote different
elements). Furthermore, together with every representative $x_j(0)\in L_j$,
$j\ne 0$, we fix all elements of its $H$-orbit
$$O(x_j(0))=\{x_j(0), x_j(0)^h,\,\ldots,
x_j(0)^{h^{q-1}} \},
$$
and also call them \textit{representatives
of level~$0$}. Elements of these orbits are denoted by
$x_{r^{s}j}(0):=x_j(0)^{h^{s}}$ under the Index Convention
(since $L_i^h\leq L_{ri}$).
\end{definition0}

The total number of patterns ${\bf p}$ of weight $\leq U$ is
$(n,c)$-bounded, $|L_0|=m$, and $|O(x_j(0))|=q$; hence the number of
representatives of level $0$ is $(m,n,c)$-bounded.

\begin{definition3}
Suppose that we have already fixed
 $(m,n,c)$-boundedly many representatives of levels $<t$.
We define \textit{graded centralizers of level $t$}
by setting, for every
 $j\ne 0$,
$$
L_j(t)=\bigcap_{\vec x} {\rm Ker}\, \vartheta_{\vec x},
$$
where $\vec x=\left( x_{i_1}(\varepsilon_1), \dots ,
x_{i_k}(\varepsilon_k)\right) $ runs over all
ordered tuples of all lengths
 $k\leq U$ consisting of representatives of (possibly different) levels
$<t$ such that
$$j+ i_1+\cdots + i_k\equiv 0\; ({\rm mod}\, n).$$

For brevity we also call elements of $L_j(t)$ \textit{centralizers of level $t$} and fix for them the notation
 $y_j(t)$ with letter ``$y$'' (under the Index Convention).

The number of representatives of all levels $<t$ is
 $(m,n,c)$-bounded and $|L_j: { \rm Ker}\,\vartheta_{\vec
x}| \leq m$ for all $\vec x$. Hence the intersection here is taken over
$ (m,n,c)$-boundedly many additive subgroups
of index $\leq m$ in $L_j$, and therefore $L_j(t)$ also has
$(m,n,c) $-bounded index in the additive subgroup $L_{j}$.

By definition a centralizer $y_j(t)$ of level
 $t$ has the following centralizer property with respect to
representatives of lower levels:
\begin{equation}\label{centralizer-property}
\left[ y_j(t), x_{i_1}(\varepsilon_1),\, \ldots ,\,
 x_{i_k}(\varepsilon_k) \right]=0,
\end{equation}
as soon as $j+ i_1+\cdots +i_k\equiv 0\, (\mbox{mod}\, n)$, \
$k\leq U$, and the elements $x_{i_s}(\varepsilon_s)$ are representatives
of any (possibly different) levels $\varepsilon_s<t$.

We now fix representatives of level $t$. For every pair $({\bf p},c)$
consisting of a pattern ${\bf p}$ of a simple commutator of weight $\leq U$
with non-zero indices of entries and zero sum of indices and a commutator $c\in L_0$
equal to the value of ${\bf p}$ on
$\varphi$-homogeneous elements \textit{of graded centralizers $L_{i}(t)$, $i\ne 0$, of level $t$},
we fix one such representation. The
elements occurring in this fixed representation of $c$ are called \textit{representatives of level $t$} and are denoted by $ x_{j}(t)$ (under the Index Convention).
Next, for every (already fixed) representative $x_j(t)$ of level $t$, we fix
the elements of the $H$-orbit
$$O(x_j(t))=\{x_j(t), x_j(t)^h,\,\ldots,
x_j(t)^{h^{q-1}} \},
$$
and call them also
 \textit{representatives of level~$t$}. These elements are denoted by
$x_{r^{s}j}(t):=x_j(t)^{h^{s}}$ under the Index Convention (since $L_j^{h^s}\leq L_{r^sj}$).
The number of patterns of weight $\leq U$ is $(n,c)$-bounded, $|L_0|=m$, and $|O(x_j(t))|=q$; hence
the total number of representatives of level~$t$ is $(m,n, c)$-bounded.
\end{definition3}

The construction of centralizers and representatives of levels $\leq T$
is complete. We now consider their properties.

It is clear from the construction of graded centralizers that
\begin{equation}\label{vkljuchenie}
L_j(k+1)\leq L_j(k)
\end{equation}
for all $j\ne 0$ and all $k=1,\ldots, T$.

The following lemma follows immediately from the definition of
representatives and from the
inclusions~\eqref{vkljuchenie}; we shall refer to this
lemma as the ``freezing'' procedure.

\begin{lemma}[freezing procedure]\label{zamorazhivanie}
Every simple commutator
$
[y_{j_1}(k_1),y_{j_2}(k_2),\dots,y_{j_w}(k_w) ]
$
of weight $w\leq U$ in centralizers of levels $k_1,k_2,\dots, k_w$ with zero modulo $n$ sum of indices
can be represented $($frozen\/$)$ as a commutator
$[x_{j_1}(s),x_{j_2}(s),\dots,x_{j_w}(s) ]$ of the same pattern in
representatives of any level $s$ satisfying $ 0\leq s\leq \min \{ k_1,k_2,\dots,k_w \}$.
\end{lemma}

\begin{definition} We define a \textit{quasirepresentative of weight $w\geq 1$ and level $k$} to be any
commutator of weight $w$ which involves exactly one
representative $x_i(k)$
 of level $k$ and $w-1$ representatives $x_s(\varepsilon_s)$ of any lower levels $\varepsilon_s<k$.
 Quasirepresentatives of level $k$ are denoted by
$\hat{x}_{j}(k)\in L_j$ under the Index Convention.
Quasirepresentatives of weight $1$ are precisely representatives.
\end{definition}

\begin{lemma}\label{invariance} If $y_j(t)\in L_j(t)$ is a centralizer of level
$t$, then $(y_j(t))^h$ is a centralizer of level $t$. If
$\hat{x}_j(t)$ is a quasirepresentative of level $t$, then
$(\hat{x}_j(t))^h$ is a quasirepresentative of level $t$ and of the same weight as $\hat{x}_j(t)$.
\end{lemma}

\begin{proof}
By hypothesis,
$$ [y_j(t), x_{i_1}(\varepsilon_1), \dots ,
 x_{i_k}(\varepsilon_k)]=0,
$$
whenever $\e _i<t$ for all $i$, $j+i_1+\cdots+ i_k \equiv 0 \; ({\rm mod}\, n)$, and $k\leq U$.
By applying the automorphism $h$ we obtain that
$$[y_{j}(t)^h, x_{ri_1}(\varepsilon_1), \dots ,
 x_{ri_k}(\varepsilon_k)]=0
$$ with $y_{j}(t)^h\in L_{rj}$. Since the set of representatives
is $H$-invariant by construction, the tuples
$x_{ri_1}(\varepsilon_1), \dots , x_{ri_k}(\varepsilon_k)$ run
over all tuples of representatives of levels $<t$ with index tuples
such that $rj+ ri_1 +\dots + ri_k\equiv 0 \; ({\rm mod}\, n)$. By
definition this means that $y_{j}(t)^h\in L_{rj}(t)$.

Now let $\hat{x}_j(t)$ be a quasirepresentative of weight $k$ of level $t$.
By definition this is a commutator
involving exactly one representative $x_{i_1}(t)$ of level $t$ and some representatives
$x_{i_2}(\varepsilon_2), \dots , x_{i_k}(\varepsilon_k)$ of smaller levels
$\varepsilon_s<t$.
By construction, the elements
$(x_{i_s}(\varepsilon_s))^{h}=x_{ri_s}(\varepsilon_s)$ are also representatives of the same
levels~$\varepsilon_s$. Then the image $(\hat{x}_j(t))^h$ is also
a commutator involving exactly one
representative $x_{ri_1}(t)$ of level $t$ and representatives
$x_{ri_2}(\varepsilon_2), \dots , x_{ri_k}(\varepsilon_k)$ of smaller levels
$\varepsilon_s<t$.
Therefore $(\hat{x}_j(t))^h$ is
 also a quasirepresentative of level $t$ of the same weight.
\end{proof}

When using Lemma~\ref{invariance} we denote the elements
$y_{j}(t)^{h^s}$ by $y_{r^sj}(t)$, and the elements $\hat{x}_j(t)^{h^s}$
by $\hat{x}_{r^sj}(t)$.

 Lemma~\ref{invariance} also implies that all representatives of level $t$,
 elements $x_j(t), x_j(t)^h,\ldots, x_j(t)^{h^{q-1}}$, are centralizers of
 level $t$.

\begin{lemma}\label{quasirepresentatives} Any commutator involving exactly one
centralizer ${y}_{i}(t)$ (or quasirepresentative ${\hat x}_{i}(t)$) of level
 $t$ and quasirepresentatives of levels $< t$ is equal to
 $0$ if the sum of indices of its entries is equal to $0$ and the sum
of their weights is at most~$U+1$.
 \end{lemma}

\begin{proof} Based on the definitions, by the Jacobi and anticommutativity identities
we can represent this
commutator as a linear combination of simple commutators of weight
 $\leq U+1$ beginning with the centralizer ${y}_{i}(t)$ (or a centralizer ${y}_{j}(t)$ involved in ${\hat x}_{i}(t)$)
of level $t$ and involving in addition only some representatives of levels
$<t$. Since the sum of indices of all these elements is also equal to
$0$, all these commutators are equal to $0$ by~\eqref{centralizer-property}.
 \end{proof}

\paragraph{Completion of the proof of the Lie ring theorem.}\label{section-main-teorem}
Recall that $T$ is the fixed notation for the highest level, which
is a $(c,q)$-bounded number. We constructed above the graded
centralizers $L_j(T)$. We now set $$ Z=\left<
L_1(T),\,L_2(T),\ldots ,L_{n-1}(T)\right> .$$
Since $|L_j:L_j(T)|$ is $(m,n,c)$-bounded for $j\ne 0$ and $|L_0|=m$,
it follows that $|L:Z|$ is $(m,n,c)$-bounded. We claim that
the subring $Z$ is nilpotent of $(c,q)$-bounded class
and therefore is a required one.

Since $Z$ is generated by the $L_{j}(T)$, $j\ne 0$,
it is sufficient to prove that every simple commutator of weight $U$ of the form
\begin{equation}\label{9}
[y_{i_1}(T),\ldots, y_{i_U}(T)],
\end{equation}
where $y_{i_j}(T)\in L_{i_j}(T)$,
 is equal to zero. Recall that the $H$-orbits
 $O(y_{i_j}(T))=\{y_{i_j}(T)^{h^s}=y_{r^si_j}(T)\mid s=0,1,\ldots, q-1\}$
 consist of centralizers of level $T$ by Lemma~\ref{invariance}. By
Proposition~\ref{kh-ma-shu-transformation}, the
commutator~\eqref{9} can be represented as a linear combination of
commutators in elements of $Y=\bigcup_{j=1}^U O(y_{i_j}(T))$ each of which either has a
subcommutator of the form~\eqref{f1}
in which there are $T$ distinct initial segments in $L_0$, or has a
subcommutator of the form~\eqref{f2} in which there are $T-1$ occurrences of
elements from $L_0$. It is sufficient to prove that such subcommutators of both types
are equal to zero.

We firstly consider a commutator of the form~\eqref{f2}
\begin{equation}\label{main-f-2}
[y_{k_0}(T), c^1,\ldots, c^{T-1}],
\end{equation}
where $y_{k_0}(T)\in Y$ and every $c^i\in L_0$ (with numbering upper indices $i=1,\ldots,T-1$)
has the form $[y_{k_1}(T),\ldots, y_{k_i}(T)]$ with $y_{k_j}(T)\in Y$ and $k_1+\cdots+k_i\equiv 0 \; ({\rm mod}\, n)$.

Using Lemma~\ref{zamorazhivanie} we ``freeze'' every $c^k$ as a
commutator of the same pattern
in representatives of level $k$.
Then by expanding the inner brackets by the Jacobi identity
$[a,[b,c]]=[a,b,c]-[a,c,b]$ we represent
the commutator~\eqref{main-f-2} as a linear combination
of commutators of the form
\begin{equation}\label{main-f-22}
[y_{k_0}(T),\,\,x_{j_1}(1),\ldots, x_{j_k}(1),\,x_{j_{k+1}}(2),\ldots,
x_{j_s}(2), \ldots, \, \, \, \,x_{j_{l+1}}(T-1),\ldots,
x_{j_{u}}(T-1)\,].
\end{equation}
We subject the commutator~\eqref{main-f-22} to a certain collecting process, aiming at
a linear combination of commutators with initial segments
consisting of (quasi)representatives of different levels $1,2,\ldots,T-1 $ and
the element $y_{k_0}(T)$. For that, by the formula
$[a,b,c]=[a,c,b]+[a,[b,c]]$,
we begin moving the element $x_{j_{k+1}}(2)$ in~\eqref{main-f-22} (the first from the left element
of level 2) to the left, in order to place it right after the element
$x_{j_{1}}(1)$. These transformations give rise to
additional summands: say, at first step we obtain
$$[y_{k_0}(T),\,\,\ldots,x_{j_{k+1}}(2),
x_{j_k}(1),\,\ldots,\,]+[u_{k_0}(T),\,\,\ldots,
[x_{j_k}(1),\,x_{j_{k+1}}(2)],\ldots].
$$
In the first summand
we continue transferring $x_{j_{k+1}}(2)$ to the left, over all
representatives of level 1. In the second summand the subcommutator
$[x_{j_k}(1),\,x_{j_{k+1}}(2)]$ is a quasirepresentative, which we denote by
$\hat{x}_{j_k+j_{k+1}}(2)$
and start moving this quasirepresentative to the left over all representatives of level 1.
Since we are transferring a (quasi)representative of level $2$ over
representatives of level $1$, in additional summands every time there appear subcommutators that
are quasirepresentatives of level $2$, which assume the role of the element being transferred.

\begin{remark}\label{rem}
Here and in subsequent similar situations, it may happen that the sum of indices of a new
subcommutator is zero, so it cannot be regarded as a quasirepresentative --- but then such a
subcommutator is equal to 0 by Lemma~\ref{quasirepresentatives}.
\end{remark}

As a result we obtain a linear combination of commutators of the form
$$[[y_{k_0}(T),x(1),\hat{x}(2)], x(1),\ldots,x(1),\,\,x(2),\ldots,
x(2),\ldots, x(T-1),\ldots,x(T-1)]
$$
with collected initial
segment $[y_{k_0}(T),x(1),\hat{x}(2)]$. (For simplicity we omitted
indices in the formula.) Next we begin moving to the left
the first from the left representative of level 3 in order to place it in the fourth place.
This element is also transferred only
over representatives of lower levels, and the new
subcommutators in additional summands are quasirepresentatives of level 3. These
quasirepresentatives of level 3 assume the role of the element being transferred, and so on.
In the end we obtain
a linear combination of commutators with initial segments of the form
\begin{equation}
\label{main-f-23}[y_{k_0}(T),\, \hat{x}_{k_1}(1),
\hat{x}_{k_2}(2),\ldots,\hat{x}_{k_{T-1}}(T-1) ].
\end{equation}

By Proposition~\ref{combinatorial} the commutator~\eqref{main-f-23} of weight
$T$ is equal to a linear combination of $\varphi$-homogeneous commutators
of the same weight $T$
in elements of the $H$-orbits of the elements
$y_{k_0}(T),\, \hat{x}_{k_1}(1),
\hat{x}_{k_2}(2),\ldots,\hat{x}_{k_{T-1}}(T-1)$
each involving exactly the same number of elements of each $H$-orbit of these elements as \eqref{main-f-23}
and having a subcommutator with zero sum of indices modulo $n$.
By Lemma~\ref{invariance} every element $(\hat{x}_
{k_i}(i))^{h^{l}}=\hat{x}_{r^lk_i}(i)$ is a quasirepresentative
 of level $i$ and any $(y_{k_0}(T))^{h^l}$
is a centralizer of the form $y_{r^lk_0}(T)$ of level $T$. Since
every level appears only once in \eqref{main-f-23}, those subcommutators with zero sum of indices are equal to 0 by
Lemma~\ref{quasirepresentatives}. Hence  every commutator of the linear combination is equal to 0.

We now consider a commutator of the form~\eqref{f1}
\begin{equation}\label{main-f-1}
[y_{k_1}(T),\ldots, y_{k_s}(T)],
\end{equation}
where $y_{k_j}(T)\in Y$ and there are $T$ distinct initial segments with
zero sum of indices: $k_1+\cdots+k_{r_i}\equiv 0\,({\rm mod}\, n)$ for $1<r_1<r_2<\cdots<r_{T}=s$.
The commutator~\eqref{main-f-1} in elements of the centralizers $L_i(T)$ of level $T$ belongs to $L_0$;
therefore by Lemma~\ref{zamorazhivanie} it can be ``frozen'' in
level $T$, that is, represented as the value of the same pattern on representatives of level $T$:
\begin{equation}\label{main-f-11}
[x_{k_1}(T),\ldots, x_{k_s}(T)].
\end{equation}
Next, the initial segment of \eqref{main-f-11} of length
$r_{T-1}$ also belongs to $L_0$ and is a commutator in centralizers of level $T-1$, since
$L_i(T-1)\leq L_i(T)$. Therefore by Lemma~\ref{zamorazhivanie} it can be
``frozen'' in level $T-1$, and so on. As a result the commutator \eqref{main-f-1} is equal to a
commutator of the form
\begin{equation}\label{main-f-12}
[x(1),\ldots, x(1),x(2),\ldots, x(2),\ldots, \ldots, x(T),\ldots,
x(T)].
\end{equation}
(We omitted here indices for
simplicity.) We subject the commutator~\eqref{main-f-12} to exactly the same
transformations as the commutator~\eqref{main-f-22}. First we transfer the left-most element
of level $2$ to the left to the second place,
then the left-most element of level 3 to the third place, and so on. In
additional summands the emerging quasirepresentatives
$\hat{x}(i)$ (see  Remark~\ref{rem})
assume the role of the element being transferred and are also
transferred to the left to the $i$th place. In the end we obtain a linear
combination of commutators with initial segments of the form
\begin{equation}\label{main-f-13}[\hat{x}_{k_1}(1),
\hat{x}_{k_2}(2),\ldots,\hat{x}_{k_{T}}(T) ].
\end{equation}
By Proposition~\ref{combinatorial} the commutator~\eqref{main-f-13} of weight
$T$ is equal to a linear combination of $\varphi$-homogeneous commutators
of the same weight $T$
in elements of the $H$-orbits of the elements
$\hat{x}_{k_1}(1), \hat{x}_{k_2}(2),\ldots,\hat{x}_{k_{T}}(T)$
each involving exactly the same number of elements of each $H$-orbit of these elements as \eqref{main-f-13}
and having a subcommutator with zero sum of indices modulo $n$.
By Lemma~\ref{invariance} every element $(\hat{x}_
{k_i}(i))^{h^{l}}$ is a quasirepresentative
$\hat{x}_{r^lk_i}(i)$ of level $i$.
Since every level appears only once in \eqref{main-f-13} and
there is an initial segment with zero sum of indices, those subcommutators with zero sum of indices  are
equal to 0 by Lemma~\ref{quasirepresentatives}.
Hence  every commutator of the linear combination is equal to 0.
\end{proof}

\begin{proof}[Proof of Corollary~\ref{c-l}]
Let $L$ be a Lie algebra over a field of characteristic $p$. Let $F=\langle \psi\rangle \times \langle
\chi\rangle$, where  $ \langle \psi \rangle$ is  the Sylow $p$-subgroup and
$\langle \chi \rangle$ the Hall $p'$-subgroup. Consider the
fixed-point subalgebra  $A=C_L(\chi )$. It is
$\psi$-invariant and $ C_A(\psi )=C_L(\varphi )$, so that
${\rm dim\,}C_A(\psi )\leq m$. Since $\psi$ has order $p^k$ and the characteristic is~$p$, this implies that
${\rm dim\,}A\leq mp^k$ by a well-known lemma following from the Jordan normal form of $\psi$, see for example,
 \cite[1.7.4]{kh4}.
Thus, $L$ admits the Frobenius group of automorphisms
$\langle\chi\rangle H$ with $(m,n)$-bounded ${\rm
dim\,}C_L(\chi )$, so we can assume that $p$ does not divide $n$.
After that we can repeat the arguments of the proof of
Theorem~\ref{t-l} with obvious modifications: codimensions instead
of indices, etc. Most significant modification is in the
definition of representatives, which now have to be fixed bases of
subspaces generated by values of patterns, rather than all values
of these patterns. Actually almost the whole proof in  \cite{mak-khu13}
can be repeated with the improvement that we made in the present paper in the proof of
Proposition~\ref{combinatorial}.
\end{proof}

\section{Proof of Theorem~\ref{t-g}}

Recall that $G$ is a finite group admitting a Frobenius group of
automorphisms $FH$ of coprime order with cyclic kernel $F$ of order $n$ and
complement $H$ of order $q$ such that the fixed-point subgroup $C_G(H)$ of the
complement is nilpotent of class $c$. Let $m=|C_G(F)|$; we need to prove that $G$ has a nilpotent
characteristic subgroup of $(m, n,c)$-bounded index and of $(c, q)$-bounded nilpotency class.

By a result of B.~Bruno and F.~Napolitani \mbox{\cite[Lemma~3]{bru-nap04}} if a group has a subgroup of finite index $k$
that is nilpotent of class $l$, then it also has a characteristic subgroup of finite $(k,l)$-bounded index
that is nilpotent of class $\leq l$. Therefore in the proof of Theorem~\ref{t-g} we only need a subgroup
of $(m, n,c)$-bounded index and of $(c, q)$-bounded nilpotency class.

By Theorem~\ref{t-n} 
the group  $G$ has a nilpotent subgroup of
$(m, n)$-bounded index. Therefore henceforth we assume that $G$ is nilpotent.

The proof of Theorem \ref{t-g} for nilpotent groups is based on
the ideas developed in \cite{kh2} for almost fixed-point-free automorphisms
of prime order. A modification of the
method of graded centralizers is employed, which was used in \S\,\ref{s-l}
for Lie rings. But now construction of fixed elements (representatives) and generalized
centralizers $A(s)$ of levels $s\leq 2T-2$ is conducted in the group $G$:
$$G=A(0)\geq A(1)\geq \dots \geq A(2T-2),$$
where $T=T(c,q)=f(c,q)+1$ and $f$ is the function in
Proposition~\ref{kh-ma-shu10-1}. The subgroups $A(s)$ will have
$(m,n,c)$-bounded indices in $G$ and the images of elements of $A(s)$
in the associated Lie ring $L(G)$ extended by a
primitive $n$th root of unity will have centralizer properties in this Lie ring
with respect to representatives of lower levels. Direct application of Theorem~\ref{t-l} to $L(G)$
does not give a required result for the group $G$, since there is no good correspondence between
subgroups of $G$ and subrings of $L(G)$. As in \cite{kh2} we overcome this difficulty by
proving that, in a certain critical situation, the group $G$ itself is nilpotent of $(c, q)$-bounded class,
the advantage being that the nilpotency class of $G$ is equal to that of $L(G)$.
This is not true in general, but can be achieved
by using induction on a certain complex parameter. This parameter controls
the possibility of replacing commutators in elements of the Lie ring by
commutators in representatives of higher levels. If the parameter becomes smaller
for some of the subgroups $A(i)$ than for the group $G$, then by the induction hypothesis the subgroup $A(i)$,
and therefore also the group $G$ itself, contains a required subgroup
of $(m,n,c)$-bounded index and of $(c,q)$-bounded nilpotency class.
If, however, this parameter does not diminish up to level $2T-2$, then we prove that
 the whole group $G$ is nilpotent of class $< V(T, 2(T-1), c,q )$, where $V$ is the
function in Proposition~\ref{kh-ma-shu-transformation}.

\paragraph{Generalized centralizers and representatives in the group.}
Let $L(G)=\bigoplus_i \g _i(G)/\g _{i+1}(G)$ be the associated Lie ring of the group $G$, where $\g _i(G)$ are
terms of the lower central series. The summand $\g _i(G)/\g _{i+1}(G)$ is the homogeneous component of weight $i$
of the Lie ring $L(G)$ with respect to the generating set $G/\g _{2}(G)$. Since
$(|G|, |FH|)=1$, we have $|C_{L(G)}(F)|=|C_G(F)|=m$ and
$C_{L(G)}(H)=\bigoplus_i
C_{\gamma_i(G)}(H)\gamma_{i+1}(G)/\gamma_{i+1}(G)$ for the induced group $FH$ of automorphisms of $L(G)$.
Therefore it is easy to see that $C_{L(G)}(H)$ is also nilpotent of class at most $c$.

Recall that $F=\langle\varphi\rangle$ is cyclic of order $n$. Let
$L=L(G)\otimes_{\Bbb Z}{\Bbb Z}[\omega]$, where $\omega$ is a
primitive $n$th root of unity. Recall that as a $\Z$-module, $\Z
[\w ]=\bigoplus _{i=0}^{E(n)-1} \w ^i \Z$, where $E(n)$ is the
Euler function. Hence,
 \be \label{euler0} L=\bigoplus
_{i=0}^{E(n)-1} L(G)\otimes \w ^i \Z.
 \ee
 In particular,
$C_L(\varphi )= \bigoplus _{i=0}^{E(n)-1} C_{L(G)}(\varphi
)\otimes \w ^i \Z$, so that $|C_L(\varphi )|=|C_{L(G)}(\varphi
)|^{E(n)}=m^{E(n)}$ is an $(m,n)$-bounded number. Since $(|G|,
n)=1$, we have $L=L_0\oplus L_1\oplus \cdots \oplus L_{n-1}$,
where $L_i=\{x\in L\mid x^{\varphi}=\w ^ix\}$ are the $\varphi$-components of $L$, as in
\S\,\ref{s-l}, and $C_L(\varphi )=L_0$. We consider $L(G)$ to be
naturally embedded in $L$ as $L(G)\otimes 1$. Since $(|G|,n)=1$,
we can assume that the ground rings of $L(G)$ and $L$ contain $1/n$.

\begin{definition} Let $x \in G$, and let $\bar{x}$ be the image of $x$ in $G/\gamma_2(G)$. We
define the $\varphi$-\textit{terms} of the element $x$
in $L$ by the formula $x_k=\frac{1}{n}\sum_{s=0}^{n-1}\omega^{-ks}\bar x^{\varphi ^s}$,
$k=0,1,\dots , n-1$. Then $x_k\in L_k$ and $\bar x=x_0+x_1+\cdots+x_{n-1}$.

We re-define \textit{$\varphi$-homogeneous commutators} as commutators in $\varphi$-terms of elements.
\end{definition}

Note that the $\varphi$-terms of elements are calculated in
the homogeneous component of weight 1 of the ring $L$ (in
particular, for elements $\gamma_2(G)$ they are all equal to $0$).
Note also that now $\varphi$-homogeneous commutators have a more narrow  meaning than in
\S\,\ref{s-l} (where they were commutators in any elements of the grading $\varphi$-components $L_i$).

As in \S\,\ref{s-l}, the group $H=\langle h\rangle$ of order $q$
permutes the $\varphi$-components $L_i$ by the rule
$L_i^h = L_{ri}$ for $i\in \Bbb Z/n\Bbb Z$. Elementary calculations show that
the action of $H$ preserves the $\varphi$-terms of elements.

\begin{lemma}\label{invariance-group} Let $x_i$, $i=0,\ldots , n-1$, be the
$\varphi$-terms of an element $x\in G$. Then $(x_{j})^h=(x^h)_{jr}$
and $(x^h)_{i}=(x_{ir^{-1}})^h$,
where $(x^h)_{k}\in L_{k}$ are the
$\varphi$-terms of $x^h\in G$.\qed
\end{lemma}

Here the construction of generalized centralizers $A(s)$ and fixed
representatives is somewhat different from how this was done in \S\,\ref{s-l}.
Complications  arise from the fact that the centralizer property
is defined in the Lie ring $L$, while the $A(s)$ are subgroups of $G$.
The following lemma interprets this property in the group $G$.

\begin{lemma}\label{svjaz}
Suppose that $j+i_1+i_2+\cdots+i_k\equiv 0\, ({\rm mod}\, n)$
for some $j,\, i_1,\, \ldots,\, i_k \in {\Bbb Z}/n{\Bbb Z}$.
Then for the $\varphi$-terms $u_j, x_{i_1},\,
y_{i_2},\, \ldots,\, z_{i_k}$ of $k+1$ elements $u,\, x,\,y,
\ldots, z \in G$ to satisfy the equation
$[u_j,\, x_{i_1},\, y_{i_2},\, \ldots,\, z_{i_k}]=0$ in the Lie ring $L$, it is sufficient
that the congruences $\prod_{t=0}^{n-1}\left[ u,\, x^{\varphi ^{a_1}},\,
y^{\varphi ^{a_2}},\, \ldots,\, z^{\varphi ^{a_k}}\right]
^{\varphi ^t}\equiv 1\, ({\rm mod}\, \gamma_{k+2}(G))$ hold in the group $G$
for all ordered tuples $a_1,\, a_2,\,
\ldots,\, a_k$ of elements of ${\Bbb Z}/n{\Bbb Z}$.
\end{lemma}

\begin{proof}
We substitute the expressions of the $\varphi$-terms:
\begin{align*}
[u_j,\, x_{i_1},\, y_{i_2},\, \ldots,\, z_{i_k}]
 &=\Big[ \frac{1}{n}\sum_{s=0}^{n-1}\omega^{-js}\bar u^{\varphi^s},\,\,
 \frac{1}{n}\sum_{s=0}^{n-1}\omega^{-i_1s} \bar{x}^{\varphi^s},\,
\ldots,\,
\frac{1}{n}\sum_{s=0}^{n-1}\omega^{-i_ks}
\bar{z}^{\varphi^s}\Big]\\
&=\frac{1}{n^{k+1}}\sum_{l=0}^{n-1} \omega^l \sum_{\substack{-js_0-i_1s_1-\cdots-i_ks_k\equiv l\, (\rm
{mod}\,n)\\ 0\leq s_i\leq n-1}}[\bar u^{\varphi^{s_0}},\,
\bar{x}^{\varphi^{s_1}},\, \bar{y}^{\varphi^{s_2}},\ldots,
\bar{z}^{\varphi^{s_k}}],
\end{align*}
where $\bar u$, $\bar{x}, \ldots, \bar{z}$ are the images of the elements $u$, $x, y,\ldots, z$ in
$G/\gamma_2(G)$ regarded as elements of the Lie ring $L$.
Since $j+i_1+\cdots+i_k\equiv 0 \, ({\rm mod}\, n) $, the summation condition
$-js_0-i_1s_1-\cdots-i_ks_k\equiv l \,(\rm {mod}\,n) $ can be rewritten as
$i_1(s_0-s_1)+i_2(s_0-s_2)+\cdots+ i_k(s_0-s_k)\equiv l \,(\rm {mod}\,n)$.
Therefore the inner sum splits into several sums of the form
$\sum_{t=0}^{n-1}\left[ \bar u,\, \bar
x^{\varphi ^{a_1}},\, \bar y^{\varphi ^{a_2}},\, \ldots,\, \bar
z^{\varphi ^{a_k}}\right] ^{\varphi ^t}$. The congruences in the statement of the lemma are equivalent to
these sums being equal to 0.
\end{proof}

We shall need homomorphisms similar to \eqref{vartheta} used in \S\,\ref{s-l} but defined on the
group $G$.
For every ordered tuple $\vec v=({x,\,y,
\ldots,\, z})$ of length $k$
of elements of $G$ and every tuple
$\vec a=(a_1,\, a_2,\, \ldots,\, a_k)$ of elements of ${\Bbb Z}/n{\Bbb Z}$
we define the homomorphism
$$
\vartheta_{\vec v,\, \vec a} :\; u\rightarrow \Big(
\prod_{t=0}^{n-1}\left[ u,\, x^{\varphi ^{a_1}},\, y^{\varphi
^{a_2}},\, \ldots,\, z^{\varphi ^{a_k}}\right] ^{\varphi^t}
\Big) \gamma _{k+2}(G)
$$
of the group $G$ into
$\gamma_{k+1}(G)/\gamma_{k+2}(G)$.

 The image of an element $u$ under
 $\vartheta_{\vec v,\, \vec a}$ is equal to the product of
commuting  elements over an orbit of the
automorphism $\varphi$ in the abelian group
$\gamma_{k+1}(G)/\gamma_{k+2}(G)$ and therefore belongs to $C_{L(G)}(\varphi
)$. Hence, $|G:\mbox{Ker}\, \vartheta_{\vec x,\, \vec a}|\leq
|C_{L(G)}(\varphi )| =m$.

We further set $K(\vec v)=\bigcap_{\vec a}\mbox{Ker}\,
\vartheta_{\vec v,\, \vec a}$, where $\vec a$ runs over all tuples
of length $k$ of elements of ${\Bbb Z}/n{\Bbb Z}$. The index of
the subgroup $K(\vec v)$ is $(m, n, k)$-bounded. A straightforward calculation shows
that $K(\vec v)^h=K(\vec v^h)$, where $\vec v^h=({x^h,y^h, \ldots,
z^h})$.

We claim that the subgroup $K(\vec v)$ is $F$-invariant.
Let
$u\in K(\vec v)$. We need to show that
$u^{\varphi}\in K(\vec v)$, that is,
$$
\prod_{t=0}^{n-1}\left[ u^{\varphi},\, x^{\varphi ^{a_1}},\, y^{\varphi
^{a_2}},\, \ldots,\, z^{\varphi ^{a_k}}\right] ^{\varphi^t}\in
\gamma_{k+2}(G).
$$
for any tuple $\vec
a=(a_1,\, a_2,\, \ldots,\, a_k)$ of elements of ${\Bbb Z}/n{\Bbb Z}$. Indeed,
$$
\prod_{t=0}^{n-1}\left[ u^{\varphi},\, x^{\varphi ^{a_1}},\, y^{\varphi
^{a_2}},\, \ldots,\, z^{\varphi ^{a_k}}\right] ^{\varphi^t}
\;\equiv\; \prod_{s=0}^{n-1}\left[ u,\, x^{\varphi
^{a_1-1}},\, y^{\varphi ^{a_2-1}},\, \ldots,\, z^{\varphi
^{a_k-1}}\right] ^{\varphi^{s}} \;({\rm mod}\,\gamma _{k+2}(G))
$$
after the substitution $s=t+1$, since the
commutators commute modulo $\gamma_{k+2}(G)$.
The right-hand side is trivial modulo  $\gamma_{k+2}(G)$, since $u\in K(\vec v)$.

By Lemma~\ref{svjaz}, for a tuple
$\vec v=({x,y,\ldots,z})$ of length $k$ the corresponding subgroup $K(\vec v)$
has the following centralizer property: for any $u\in K(\vec v)$ and
for the $\varphi$-terms of the elements in $\vec v$,
\be\label{c-prop}
[u_j,\, x_{i_1},\, y_{i_2},\, \ldots,\, z_{i_k}]=0
\ee
in the Lie ring $L$  as soon as $j+i_1+i_2+\cdots+i_k\equiv 0\, (\mbox{mod}\, n)$.

\begin{notation}
 For what follows we fix the notation  $N=N(c,q)=V(T, 2(T-1),
c,q)$, where $V$ is the functions
in~\ref{kh-ma-shu-transformation}.
\end{notation}

We now begin the construction of generalized centralizers $A(i)$ of levels
$i\leq 2T-2$ with simultaneous fixation of representatives both in the group $G$
and in the homogeneous component of weight 1
of the Lie ring $L$. The  level is indicated in parenthesis. We set $A(0)=G$.
Recall that the \textit{pattern} of a commutator in elements of the $L_i$
is its bracket structure together with the arrangement of the indices.

\begin{definition0}
At level 0 we only fix
representatives of level $0$. For every pair
 $({\bf p},c)$ consisting of a pattern ${\bf p}$ of a simple
 $\varphi$-homogeneous commutator of weight $\leq N$ with nonzero indices and zero sum of indices and a commutator
 $c\in L_0$ equal to the value of this pattern on the $\varphi$-terms of elements,
we fix one such representation. The $\varphi$-terms $a_j$ of elements $a\in G$
(which belong to the homogeneous component of weight 1 of $L$)
occurring in this fixed
representation of the commutator $c$, as well as these elements $a\in G$ themselves, are called
 \textit{ring and group representatives of level~$0$} and are denoted by $x_j (0)\in L_j$ (under the Index Convention)
and $x(0)\in G$, respectively.

Together with every ring representative $x_j(0)\in L_j,$  $j\ne
0$, we fix all elements of its $H$-orbit
$O\left(x_j(0)\right)=\{x_j(0), x_j(0)^h,\,\ldots,
x_j(0)^{h^{q-1}} \}$, as well as all elements of $G$ in the
$H$-orbit $O\left(x(0)\right)=\{x(0), x(0)^h,\,\ldots,
x(0)^{h^{q-1}} \}$ of the corresponding group representative,
which we also call (ring and group)  \textit{representatives of
level~$0$}. Elements of the orbit $O(x_j(0))$ are denoted by
$x_{r^{s}j}(0):=x_j(0)^{h^{s}}$ under the Index Convention (since
$L_i^h\leq L_{ri}$). By Lemma~\ref{invariance-group} the elements
$x_{r^{s}j}(0)=x_j(0)^{h^{s}}$ are the $\varphi$-terms of the
element $x(0)^{h^s}$; therefore all ring representatives are
$\varphi$-terms of some group representatives.
 Since the total number of patterns ${\bf p}$ of weight $\leq N$
is $(n,c)$-bounded, $|L_0|\leq m^{E(n)}$, and every $H$-orbit has size~$q$,
it follows that the number of  representatives of level $0$ is $(m,n,c)$-bounded.
\end{definition0}

\begin{definition3}
 Suppose that we already fixed
 $(m,n,c)$-boundedly many representatives of levels $s=0,\dots ,t-1$, both elements of the group $x(s)\in A(s)$ and
their $\varphi$-terms $x_j(s)$. Suppose also that the set of representatives is $H$-invariant.

We now define \textit{generalized centralizers of level $t$} (or,
in brief, \textit{centralizers of level $t$})  setting
$A(t)=\bigcap_{\vec x}K(\vec x)$, where $\vec
x=\left(x^1(\varepsilon_1),\, \ldots,\, x^k(\varepsilon_k)\right)
$ runs over all ordered tuples of lengths $k$ for all $k\leq N$
composed of group representatives $x^s(\varepsilon_s)\in
A(\varepsilon)$ of levels $\varepsilon_s<t$. Here we use numbering
upper indices, since lower indices always indicate the belongness
to
 $\varphi$-components of the Lie ring.

We call elements $a\in A(t)$, as well as their
$\varphi$-terms $a_j$, \textit{group and ring centralizers of level
$t$} and fix for them the notation $y(t)$ and $y_j(t)$, respectively (under the Index Convention)
indicating level in parentheses.

Clearly, $A(t)\leq A(t-1)$.
Note that the subgroup $A(t)$ is $F$-invariant, since all the subgroups $K(\vec x)$ are
$F$-invariant. We claim that $A(t)$ is also $H$-invariant.
If $y(t)\in A(t)$, then $y(t)\in K(\vec v)$
for any tuple $\vec v=\left(x^1(\varepsilon_1),\,
\ldots,\, x^k(\varepsilon_k)\right) $ of
length $k\leq N$ composed of representatives of levels $\varepsilon_j<t$. Then
$y(t)^h\in K(\vec v)^h=K(\vec v ^h)$. Since the set of
representatives of levels $<t$ is $H$-invariant, the tuples $\vec v ^h$ also run over
all tuples of lengths $k\leq N$ composed of representatives of levels $<t$.
Hence, $y(t)^h\in A(t)$.

Since the number of representatives of levels $<t$ is
$(m,n,c)$-bounded, the intersection $A(t)=\bigcap_{\vec x}K(\vec x)$ is taken over
$(m,n,c)$-boundedly many subgroups of $(m,n,c)$-bounded
index and therefore also has $(m,n,c)$-bounded index in $G$.

 Note that by \eqref{c-prop} ring centralizers
 of level $t$ have the following centralizer property with respect to
representatives of lower levels $\varepsilon_i<t$:
\begin{equation}\label{group-centralizer-property}
\left[ y_j(t), x_{i_1}(\varepsilon_1),\, \ldots, x_{i_k}(\varepsilon_k) \right]=0,
\end{equation}
as soon as $k\leq N$ and $j+ i_1+\cdots+ i_k\equiv 0\, ({\rm
mod}\, n)$. (No numbering indices here under the Index Convention,
so the $x_{i_k}(\varepsilon_k)$ with the same index may be the
$\varphi$-terms of different elements $x(\varepsilon_k)$.)

We now fix representatives of level $t$. For every pair
 $({\bf p},c)$ consisting of a pattern ${\bf p}$ of a simple
 $\varphi$-homogeneous commutator of weight $\leq N$ with
nonzero indices and zero sum of indices and a commutator
 $c\in L_0$ equal to the value of this pattern on ring centralizers of level $t$ (which are
 $\varphi$-terms $y_j(t)$ of elements $y(t)\in A(t)$),
we fix one such representation. The $\varphi$-terms $a_j$ of elements $a\in G$
(which belong to the homogeneous component of weight 1 of $L$)
occurring in this fixed
representation of the commutator $c$, as well as these elements $a\in G$ themselves, are called
\textit{ring and group representatives of level~$t$} and are denoted by $x_j (t)\in L_j$ (under the Index Convention)
and $x(t)\in G$, respectively.

Together with every ring representative $x_j(t)\in L_j$, $j\ne 0$,
we fix all elements of its $H$-orbit $O(x_j(t))=\{x_j(t),
x_j(t)^h,\,\ldots, x_j(t)^{h^{q-1}} \}$, as well as all elements
of $G$ in the $H$-orbit $O(x(t))=\{x(t), x(t)^h,\,\ldots,
x(t)^{h^{q-1}} \}$ of the corresponding group representative,
which we also call (ring and group) \textit{ representatives of
level~$t$}. Elements of the orbit $O(x_j(t))$ are denoted by
$x_{r^{s}j}(t):=x_j(t)^{h^{s}}$ under the Index Convention (since
$L_i^h\leq L_{ri}$). By Lemma~\ref{invariance-group} the elements
$x_{r^{s}j}(t)=x_j(t)^{h^{s}}$ are the $\varphi$-terms of the
element $x(t)^{h^s}$; therefore all ring representatives are
$\varphi$-terms of some group representatives.

 Since the total number of patterns ${\bf p}$ of weight $\leq N$
is $(n,c)$-bounded, $|L_0|\leq m^{E(n)}$, and every
$H$-orbit has size~$q$, it follows that the number of
representatives of level $t$ is $(m,n,c)$-bounded.
\end{definition3}

The construction of generalized centralizers and representatives of levels
$\leq 2T-2$ is complete.
It is important that ring centralizers and representative ``in the new sense'' enjoy similar
properties as
graded centralizers and representatives defined in \S\,\ref{s-l}. In particular, we can
``freeze'' commutators in ring centralizers as
commutators in ring representatives of the same or any lower level, that is, an analogue of
Lemma~\ref{zamorazhivanie} holds. Quasirepresentatives (in the Lie ring) are defined in the same fashion as
in \S\,\ref{s-l}.
The following lemma is an analogue of Lemma~\ref{invariance}.

\begin{lemma}\label{invariance-group-2}
If $y_j(t)$ is a ring centralizer of level
$t$, then $y_j(t)^h$ is a centralizer of level $t$. If
$\hat{x}_j(t)$ is a quasirepresentative of level $t$, then $(\hat{x}_j(t))^h$ is a quasirepresentative of
level~$t$.
\end{lemma}

\begin{proof}
The assertion of the lemma  are proved by repeating word-for-word
the proof of Lemma~\ref{invariance}.
\end{proof}

Since the centralizer property \eqref{group-centralizer-property}
holds for commutators of weight $\leq N+1$, rather than $\leq
U+1$ as in \eqref{centralizer-property}, in an analogue of Lemma
\ref{quasirepresentatives} the weight parameter $U+1$ must be
changed to $N+1$.

\begin{lemma}\label{group-quasirepresentatives}
Any commutator involving exactly one ring centralizer ${y}_{i}(t)$ (or quasirepresentative $\hat
{x}_{i}(t)$) of level $t$ and quasirepresentatives of lower  levels $< t$ is equal to
 $0$ if the sum of indices of its entries is equal to $0$ and the sum
of their weights is at most~$N+1$. \end{lemma}

\begin{proof}
The proof repeats word-for-word the proof of Lemma~\ref{quasirepresentatives} with the
centralizer property \eqref{centralizer-property} replaced by \eqref{group-centralizer-property}.
 \end{proof}

\paragraph{Induction parameter.}
In contrast to Theorem~\ref{t-l}, where we proved the nilpotency of the Lie subring
generated by centralizers of maximal level, in Theorem~\ref{t-g} we prove that the Lie ring $L$ itself
is nilpotent of bounded class (in a certain critical situation). Therefore here we must consider
commutators in arbitrary $\varphi$-terms of elements. The following parameter
enables us to control the possibility of replacing $\varphi$-homogeneous subcommutators in
$L_0$ by the values of the same patterns on representatives of higher
levels.

 \begin{definition}
 The \textit{induction parameter} is defined to be the triple $(m,\, \bar m,\, t)$,
where $m=|C_G(\varphi )|$;
$\bar m=(m_1, m_2, \ldots, m_N)$ for
$m_j=|C_{\gamma_j(G)/\gamma_{j+1}(G)}(\varphi )|$;
and $t=|{\mathscr P}(G)|$, where ${\mathscr P}(G)$
is the set of all pairs $({\bf p}, c)$ consisting of a pattern ${\bf p}$ of a weight $\leq N$ with nonzero
indices of entries  and zero  sum of indices  and a commutator $c\in L_0$ equal to the value of ${\bf p}$
on $\varphi$-terms $a_j$ of elements $a\in G$.\end{definition}

We denote by $(m(B), \bar m(B), t(B))$ the triple constructed in the same fashion for
an $ F$-invariant
subgroup $B$. In particular, if $M=L(B)\otimes_{\Bbb Z}\Bbb Z[\omega]$ and
$M=M_0\oplus M_1\oplus \cdots \oplus M_{n-1}$ is the decomposition into the direct sum of
$\varphi$-components, then ${\mathscr P}(B)$ is the set of pairs
$({\bf p}, c)$ consisting of a pattern ${\bf p}$ of weight $\leq N$ with nonzero
indices of entries  and zero  sum of indices  and a commutator $c\in M_0$ equal to the value of ${\bf p}$
on $\varphi$-terms $b_j$ in $M$ of elements $b\in B$ (defined in the same fashion as $\varphi$-terms for $G$ and $L$).

We introduce the inverse lexicographical order on the vectors
$\bar m = (m_1, m_2, \ldots, m_N)$:
$$(m_{11}, m_{12}, \ldots, m_{1N}) <
(m_{21}, m_{22}, \ldots, m_{2N}) \Leftrightarrow$$
$$\Leftrightarrow \text{ for some } k \geq 1 \; \; m_{1i} = m_{2i}
\text{ for all } i < k \text{ and } m_{1k} > m_{2k}.$$

We introduce the lexicographical order on the triples $(m,\, \bar m,\, t)$:
\begin{align*} (m_1, \bar m_1, t_1) < (m_2, \bar m_2, t_2) \Leftrightarrow &  \text{ either }m_1 < m_2, \\
& \text{ or } m_1 = m_2\text{ and } \bar m_1 < \bar m_2, \\
& \text{ or } m_1 = m_2,\, \, \, \bar m_1 = \bar m_2,
\text{ and } t_1 < t_2.
\end{align*}

\begin{lemma}\label{nerav}
For any $\varphi$-invariant subgroup $B\leq G$ we have
$(m(B), \bar m(B), t(B))\leq (m(G), \bar m(G), t(G))$ with respect to the order introduced above.
\end{lemma}

\begin{proof}
Clearly, $m(B)\leq m(G)$. Now suppose that $m(B)=m(G)$, that is,
$C_B(\varphi)= C_G(\varphi)$; we claim that then $\bar m(B)\leq
\bar m(G)$. Suppose that for some  $1\leq k\leq N$ we  have
$m_i(B)=m_i(G)$ for all $i<k$ (this is vacuous for $k=1$); we need
to show that $m_k(B)\geq m_k(G)$. Since
$|C_G(\varphi)|=|C_B(\varphi)|$ and $m_i(B)=m_i(G)$ for all $i<k$,
we have $|C_{\gamma_k(B)}(\varphi)|=|C_{\gamma_k(G)}(\varphi)|$.
Since $C_{\gamma_k(B)}(\varphi) \leq C_{\gamma_k(G)}(\varphi)$,
these subgroups coincide. Let 
$D:=C_{\gamma_k(B)}(\varphi)= C_{\gamma_k(G)}(\varphi)$. Since
$(|G|, n)=1$, we have $C_{\gamma_{k}(B)/\gamma_{k+1}(B)}(\varphi)=
D \gamma_{k+1}(B)/\gamma_{k+1}(B)\cong D/D\cap \gamma_{k+1}(B)$,
as well as $C_{\gamma_k(G)/\gamma_{k+1}(G)}(\varphi)= D
\gamma_{k+1}(G)/\gamma_{k+1}(G)\cong D/D\cap \gamma_{k+1}(G)$.
Clearly, $|D/D\cap \gamma_{k+1}(B)|\geq |D/D\cap
\gamma_{k+1}(G)|$. Hence, $m_k(B)\geq m_k(G)$.

 Finally, suppose
that $\bar m(B)=\bar m(G)$; we claim that then $t(B)\leq t(G)$. We
saw above that then $C_G(\varphi)\cap \gamma_k(B)=C_G(\varphi)\cap
\gamma_k(G)$ for all  $k\leq N$. Hence  for every $k\leq N$,
$$
C_{\gamma_k(B)/\gamma_{k+1}(B)}(\varphi)\cong
\left(C_G(\varphi)\cap \gamma_k(B)\right)/\left(C_G(\varphi)\cap
\gamma_{k+1}(B)\right)=$$ $$=\left(C_G(\varphi)\cap
\gamma_k(G)\right)/\left(C_G(\varphi)\cap
\gamma_{k+1}(G)\right)\cong
C_{\gamma_k(G)/\gamma_{k+1}(G)}(\varphi).
 $$
For $k\leq N$, let $\bar c$ and $\tilde c$ denote the images of an
element $c\in C_{\g _k(G)}(\varphi)=C_{\g _k(B)}(\varphi)$ 
 in
$\gamma_{k}(G)/\gamma_{k+1}(G)$ and
$\gamma_{k}(B)/\gamma_{k+1}(B)$, respectively. Clearly, the
isomorphism $C_{\gamma_k(B)/\gamma_{k+1}(B)}(\varphi)\cong
C_{\gamma_k(G)/\gamma_{k+1}(G)}(\varphi)$ is induced by the
mapping
 $\tilde c\to\bar c$.

It is also clear that the same mapping induces the natural
isomorphism
 $$
 \sigma _k: C_{\gamma_k(B)/\gamma _{k+1}(B)}(\varphi
)\otimes _\Bbb Z \, \Bbb Z [\omega]\;\to \;
C_{\gamma_k(G)/\gamma_{k+1}(G)}(\varphi )\otimes _\Bbb Z \, \Bbb Z
[\omega].
 $$
In view of the decompositions
\be \label{euler}
L=\bigoplus_{i=0}^{E(n)-1}
(L(G)\otimes \w ^i)\quad \text{and}\quad  M=\bigoplus_{i=0}^{E(n)-1} (L(B)\otimes
\w ^i),
\ee
where $E(n)$ is the Euler function (see \eqref{euler0}), the isomorphism $\sigma _k$ is
given by the mapping
 $$
\sigma _k:\; \sum _{i=0}^{E(n)-1}\tilde c ^i\w ^i \;\to\; \sum _{i=0}^{E(n)-1}\bar c ^i\w ^i,
 $$
where $c^0,c^1,\dots ,c^{E(n)-1}$ with numbering upper indices are elements of $C_{\g _k(G)}(\varphi)=C_{\g
_k(B)}(\varphi)$.

Now suppose that an element $\sum _{i=0}^{E(n)-1}\tilde c ^i\w ^i \in C_{\gamma_k(B)/\g
_{k+1}(B)}(\varphi )\otimes _\Bbb Z \, \Bbb Z [\omega]$ is equal to the value of a
pattern ${\bf p}$ 
with nonzero
indices and zero sum of indices  of weight $k\leq N$ in $\varphi$-terms
of some elements  of  $B/\gamma_2(B)$ regarded as elements of the Lie ring
$M$. This means  that
\begin{equation}\label{group-f-1}
\sum _{i=0}^{E(n)-1}\tilde c ^i\w ^i = \varkappa (\tilde
b^1_{i_1},\ldots, \tilde b^k_{i_k}),
\end{equation} where $\ka$ is a
$\varphi$-homogeneous commutator of weight $k$
 with nonzero
indices and zero sum of indices in the $\varphi$-terms $\tilde
b^1_{i_1},\ldots, \tilde b^k_{i_k}$  in $M$
of elements $b^1,\ldots, b^k\in B$ with numbering upper indices.
(We use tildes to denote the $\varphi$-terms in $M$
to distinguish them from $\varphi$-terms of the same elements
$b^1,\ldots, b^k$ with respect to $G$ and $L$.)
In view of \eqref{euler} the equality of two
elements of the Lie ring $M$
is equivalent, after  collecting terms,
to the equalities of the coefficients of $1, \omega, \ldots,
\omega^{E(n)-1}$ --- the coefficients which are elements of the Lie ring
$L(B)$.

 Since the $\varphi$-terms
$\tilde b^s_{i_s}\in L_{i_s}$ in \eqref{group-f-1} are canonically expressed with coefficients in $\Z [\w ]$
in terms of the images in $B/\gamma_2(B)$ of the elements
$(b^s)^{\varphi ^j}\in B$, equation \eqref{group-f-1} in $M$ is equivalent to a certain
system of congruences of group commutators $\ka ^{\alpha}$ (with numbering upper indices) of weight $k$ in the
elements $(b^s)^{\varphi^j}\in B$:
\begin{equation}\label{systema}
 \left\{
\begin{array}{ccl}
c^0&\equiv &\prod_{\alpha} \ka^{\alpha} \;({\rm mod}\, \gamma_{k+1}(B))\\
c^1&\equiv &\prod_{\alpha}  \ka^{\alpha}\; ({\rm mod}\, \gamma_{k+1}(B))\\
\vdots && \\

c^{E(n)-1}&\equiv &\prod_{\alpha}  \ka^{\alpha}\; ({\rm mod}\,
\gamma_{k+1}(B))
\end{array}
\right.
\end{equation}

We now define a mapping
\begin{equation}\label{nu}
\nu _k: ({\bf p}, \tilde C)\rightarrow ({\bf p},\bar C),
\end{equation}
where $\tilde C = \sum _{i=0}^{E(n)-1}\tilde c ^i\w ^i$ for $c^{i}\in C_{\g _k(G)}(\varphi)=C_{\g
_k(B)}(\varphi)$, and  $\bar C = \sigma _k(\tilde C)=\sum _{i=0}^{E(n)-1}\bar c ^i\w ^i$ for the same elements
$c^{i}$.

\bl \label{same} If $\bar m(B)=\bar m(G)$, then $\nu _k(({\bf
p},\tilde C))\in {\mathscr P}(G)$ and, moreover, $\bar C$ is the
value of the same pattern ${\bf p}$ on the $\varphi$-terms
$b^1_{i_1},\ldots, b^k_{i_k}$ in the Lie ring $L$ of the same
elements $b^1,\ldots, b^k$ in \eqref{group-f-1}, that is, $\bar
C=\ka (b^1_{i_1},\ldots, b^k_{i_k}) $ for the same commutator $\varkappa$ as in \eqref{group-f-1}.
\el

\bp Indeed, the system of congruences \eqref{systema} remains
valid if we replace $\gamma_{k+1}(B)$ by a larger subgroup
$\gamma_{k+1}(G)$. But modulo $\gamma_{k+1}(G)$ this system is
equivalent to the required equation $\bar C= \varkappa (b^1_{i_1},\ldots, b^k_{i_k}) $
  in $L$. Thus, the pair
$({\bf p},\bar C)=\nu _k({\bf p},\tilde C)$ belongs to  ${\mathscr
P}(G)$. \ep

We now complete the proof of Lemma~\ref{nerav}. As we saw above, if $\bar m(B)=\bar m(G)$ for all
$k\leq N$, then the  mapping $\sigma _k:\tilde C\to \bar C$ is an isomorphism. Hence the union
$\nu =\bigcup _{k=1}^N\nu _k$ of the mappings
\eqref{nu}
is an injective  mapping of ${\mathscr P}(B)$ into ${\mathscr P}(G)$. Thus, $t(B)\leq t(G)$ if $\bar m(B)=\bar m(G)$.
\end{proof}

It follows from the above that if $(m(B), \bar
m(B), t(B))=(m(G), \bar m(G), t(G))$, then $\nu$
is a one-to-one correspondence between the sets
${\mathscr P}(B)$ and ${\mathscr P}(G)$. Applying this to the centralizers $A(i)$ we obtain the following.

\begin{corollary} \label{group-cor}
Suppose that $(m(A(i)), \bar m(A(i)), t(A(i)))=(m(G),
\bar m(G), t(G))$ for all levels $i=1,\ldots, 2T-2$.  Then
every simple $\varphi$-homogeneous commutator in $C_L(\varphi)=L_0$
equal to the value of a pattern ${\bf p}$ with nonzero indices of weight $k\leq N$ on $\varphi$-terms of some
elements of $G$ can be represented as the value of the same pattern ${\bf p}$
on representatives of level $s$ for every
$s=0, 1,\ldots, 2T-2$.
\end{corollary}

\begin{proof}
Suppose that $\bar C\in C_{\gamma_k(G)/\g _{k+1}(G)}(\varphi
)\otimes _\Bbb Z \, \Bbb Z [\omega]$ is equal to the value of the
pattern ${\bf p}=[*_{i_1},\ldots,*_{i_k}]$ with nonzero indices
and zero sum of indices of weight $k\leq N$ on $\varphi$-terms of
some elements. Since $(m(A(s)), \bar m(A(s)), t(A(s)))=(m(G), \bar
m(G), t(G))$,  the mapping $\nu$ is a one-to-one correspondence
between ${\mathscr P}(B)$ and ${\mathscr P}(G)$. Hence the pair
$({\bf p}, \bar C)$ is the  image under $\nu _k$ of a pair $({\bf
p}, \tilde C)$ for the same pattern ${\bf p}$ and for $\tilde
C=[\tilde g^1_{i_1},\ldots, \tilde g^k_{i_k}]$ in $M$, where
$\tilde g^j_{i_t}$ are $\varphi$-terms in $M=L(A(s))\otimes _\Bbb
Z \, \Bbb Z [\omega]$ of elements $g^j\in A(s)$. Moreover,  by
Lemma~\ref{same} then $\bar C=[g^1_{i_1},\ldots,  g^k_{i_k}]$ in
$L$, where $g^j_{i_t}$ are the corresponding $\varphi$-terms in
$L$ of the same elements $g^j\in A(s)$. Since the elements $g^j$
belong to the generalized centralizer $A(s)$ (and therefore should
be denoted $y^j(s)=g^j$),
by the construction of
representatives of level $s$ the commutator $\bar
C=[y^1_{i_1}(s),\ldots, y^k_{i_k}(s)]$ can be ``frozen'' in
level $s$, that is, represented as the value of the same pattern
on fixed ring representatives $\bar C =[x^1_{i_1}(s),\ldots,
x^k_{i_k}(s)]$.
\end{proof}

\paragraph{Completion of the proof of Theorem \ref{t-g}}
Note that for a given value of  $|C_G(\varphi)|=m(G)=m$ the number of
possible triples $(m(G),\, \bar m(G),\, t(G))$ is obviously
$(m,n,c)$-bounded. Therefore we can use induction on the
parameter $(m(G), \bar m(G), t(G))$ in order to show that the
group $G$ contains a subgroup of $(m,n,c)$-bounded index
that is nilpotent of $(c,q)$-bounded class  $< N$. The basis of
induction is the case $m(G)=1$, which means that $C_G(F)=1$; then
 the group $G$ is nilpotent of class $\leq f(c,q)<N$ by the
Makarenko--Khukhro--Shumyatsky Theorem~\cite{khu-ma-shu}.

If for some $i=1,\ldots ,2T-2$ the induction parameter for
the subgroup $A(i)$ becomes smaller, that is,
$(m(A(i)), \bar m(A(i)), t(A(i)))<(m(G), \bar m(G), t(G))$, then by the induction hypothesis applied to
the $FH$-invariant subgroup $A(i)$ it contains a subgroup of $(m,n,c)$-bounded index in $A(i)$
that is nilpotent of class $<N$, which is a required subgroup, since the index of
$A(i)$ in $G$ is also $(m,n,c)$-bounded.

Therefore it is sufficient to consider the case where $(m(A(i)), \bar m(A(i)), t(A(i)))=(m(G),
\bar m(G), t(G))$  for all  $i=1,\ldots 2T-2$, and we assume this in what follows.
We claim that in this critical situation the whole group is nilpotent
of class $<N$. For that it is sufficient to show that the Lie ring $L$
is nilpotent of class $<N$.

Note that we can also assume that
$C_G(F)\leq \gamma_2(G)$, since in the opposite case
$C_{[G,F]}(F)<C_G(F)$, and then the result follows by the induction hypothesis
applied to the $FH$-invariant subgroup $[G,F]$, whose index is $\leq m$, since $G=[G,F] C_G(F)$.
Therefore also $C_L(\varphi)\leq \gamma_2(L)$ and the Lie ring $L$
is generated by $\varphi$-terms $a_i\in L_i$ of elements for
$i\neq 0$. Since the nilpotency identity can be verified on the generators of the Lie ring,
it is sufficient to show that
\begin{equation}\label{group-f-4}
[a_{i_1}, \ldots, a_{i_N}]=0
\end{equation}
 for any $\varphi$-terms  $a_{i_s}\in L_{i_s}$, $i_s\neq 0$, of elements of $G$.
 For that, in turn, it is sufficient to show
 the triviality of all commutators of the form~\eqref{f1} and~\eqref{f2} in
Proposition~\ref{kh-ma-shu-transformation} applied to
the commutator~\eqref{group-f-4} with $t_1=T$ and $t_2=2(T-1)$, where, recall,
$N=V(T, 2(T-1), c,q )$ and $T=f(c,q)+1$ for the function $f(c,q)$ in
Proposition~\ref{kh-ma-shu10-1}.

We now repeat, almost word-for-word, the final arguments in the proof of
Theorem~\ref{t-l}, with obvious replacement of the
centralizer property relative to representatives in the old sense by the similar property in
the new sense. This is possible, since in the critical situation under consideration
Corollary~\ref{group-cor} guarantees the possibility of representing elements in $L_0$ that
are commutators in  $\varphi$-terms of elements in the form of the values of the same patterns in
representatives of any level $\leq 2T-2$. A small modification of the arguments, which required
replacing $T-1$ by $2(T-1)$ as  the value   of the parameter $t_2$  in
Proposition~\ref{kh-ma-shu-transformation}, is only needed for a commutator of the form~\eqref{f2}.

First we consider a  commutator
\begin{equation}\label{group-f-5}
[u_{k_1},\ldots, u_{k_s}],
\end{equation}
of the form~\eqref{f1}
having $T$ distinct initial segments with zero
sum of indices modulo~$n$, that is, with $k_1+k_2+\cdots+k_{r_i}\equiv 0\; ({\rm mod}\, n)$ for
$1<r_1<r_2<\cdots<r_{T}=s$. Since
$(m(A(i)), \bar m(A(i)), t(A(i)))=(m(G),
\bar m(G), t(G))$ for all $i=1,\ldots, 2T-2$, by Corollary~\ref{group-cor} we can represent the
commutator~\eqref{group-f-5} as the value of the same pattern on representatives of level $T$.
Then, using the inclusions $A(i)\geq A(i+1)$, we successively represent
the initial segments of lengths  $r_{T-1}, r_{T-2},\ldots $ of the resulting commutator as the value of the same
pattern on representatives of levels $T-1, T-2, \ldots$
(since all these initial segments, as well as the
commutator~\eqref{group-f-5} itself, belong to $L_0$). As a result we obtain a
commutator equal to~\eqref{group-f-5} and having the form
$$
[x(1),\ldots, x(1),x(2),\ldots, x(2),\ldots, \ldots, x(T),\ldots,
x(T)],
$$
where we omitted indices to lighten the notation. We apply to this
commutator the same, word-for-word,  arguments that prove the equality to 0 of the
commutator~\eqref{main-f-12}. We only need to replace the centralizer
property~\eqref{centralizer-property} and Lemma~\ref{quasirepresentatives} used in \S\,\ref{s-l} by the
centralizer property~\eqref{group-centralizer-property} and Lemma~\ref{group-quasirepresentatives}.

We now consider a commutator
\begin{equation}\label{group-f-6}
[a_j, c^1_0,\ldots ,c^{2T-2}_0]
\end{equation}
of the form~\eqref{f2}, where $c^i_0\in L_0$ with numbering upper
indices are simple $\varphi$-homogeneous commutators in
$\varphi$-terms in $L_j$, $j\neq 0$, of elements of $G$. For each
$s=1,\ldots,2T-2$ we substitute into~\eqref{group-f-6} an
expression of the commutator $c^s_0$ as the value of the same
pattern of weight $<N$ on representatives of level $s$, which is
possible by Corollary~\ref{group-cor}.
After expanding
the inner brackets we obtain a linear combination of commutators
of the form
\be\label{bad}
[a_j, x(1),\ldots, x(1),\,\ldots, x(2T-2),\ldots,
x(2T-2)].
 \ee
In contrast to the commutator~\eqref{main-f-22} in
\S\,\ref{s-l}, here $a_j$ does not necessarily belong to a centralizer of high level, so we need a different argument.

The same arguments as those applied above to the
commutator~\eqref{main-f-22} make it possible to represent \eqref{bad} as
a linear combination of commutators with initial segments  of the
form
 \be\label{dvoinoi}
 [a_j, \hat x(1), \hat x(2)\ldots, \hat x(2T-2)],
 \ee
 where $a_j$ is followed by  quasirepresentatives of
levels $1,\ldots, 2T-2$, one in each level.
By Proposition~\ref{combinatorial} the initial segment $[a_j, \hat x(1),\hat x(2),\,\ldots, \hat x(T-1)]$
of  weight $T$
 of the
commutator \eqref{dvoinoi} is equal to a linear combination of simple
$\varphi$-homogeneous commutators in elements of the $H$-orbits
of the elements $a_j, \hat x(1),\hat x(2),\,\ldots, \hat x(T-1)$
each involving exactly the same number of elements of each $H$-orbit of these elements as \eqref{dvoinoi}
and having an initial segment in $L_0$. By Lemma~\ref{invariance-group-2} each
element $\hat{x}_ {k_i}(i)^{h^{l}}$ is a
quasirepresentative of the form $\hat{x}_{r^lk_i}(i)$ of level $i$.
If such an initial segment does not contain an element of the
$H$-orbit of $a_j$, then this initial segment is equal to 0
by Lemma~\ref{group-quasirepresentatives}, since the levels are all different.
If, however, this initial segment in $L_0$ does contain
an  element of the $H$-orbit  of $a_j$, then we can freeze it in level 0, that is, replace by the value of the same
pattern on representatives of level $0$. Therefore it remains to prove equality to zero of
a  commutator of the form
\begin{equation}\label{group-f-7}
[x(0),\ldots, x(0), \hat x(\varepsilon_{r+1}),\ldots,\hat x(\varepsilon_{s}), x(T),\ldots,
x(T),\ldots, x(2T-2),\ldots, x(2T-2)],
\end{equation}
which now involves only (quasi)representatives, and the  levels
$\varepsilon_{r+1},\ldots,\varepsilon_{s}$ are all less than~$T$.
We now apply almost the same collecting process as the one applied
above to~\eqref{main-f-22}. The difference is that we move to
the left only (first from the left) quasirepresentatives  of
levels $\geq T$, and collected parts are initial segments of the
form $[x(0), \hat x(T),\ldots, \hat x(T+s)]$. As a result the
commutator~\eqref{group-f-7} becomes equal to a linear combination
of commutators in quasirepresentatives with initial segments of
length $T$ of the form $[x(0), \hat x(T),\ldots, \hat x(2T-2)]$.
By applying  Proposition~\ref{combinatorial} to such an initial
segment we obtain a linear combination of $\varphi$-homogeneous
commutators in elements of the $H$-orbits of the  elements $x(0), \hat x(T),\ldots, \hat
x(2T-2)$, each involving exactly the same number of elements of each $H$-orbit of these elements as that initial segment and having an initial segment in $L_0$.
By  Lemma~\ref{invariance-group-2} elements of the $H$-orbits of the  elements $x(0), \hat x(T),\ldots, \hat
x(2T-2)$
are also  quasirepresentatives of the same levels. These initial segments in $L_0$
are equal to $0$ by Lemma~\ref{group-quasirepresentatives}, since
the levels are all different. \qed

\end{document}